\documentclass[graybox]{svmult}  

\usepackage{type1cm}        
\let\latexput\put
\usepackage{pictex}
\let\put\pictexput
\let\put\latexput
\usepackage{makeidx}         
\usepackage{graphicx}        
\usepackage{multicol}        
\usepackage[bottom]{footmisc}

\usepackage[varvw]{newtxmath}       
\usepackage{amsrefs}

\newcommand{\Z}{{\mathbb Z}}
\newcommand{\R}{{\mathbb R}}

\newcommand{\C}{{\mathbb C}}

\newcommand{\CA}{{\mathcal A}}

\newcommand{\CB}{{\mathcal B}}
\newcommand{\CE}{{\mathcal E}}
\newcommand{\CF}{{\mathcal F}}
\newcommand{\CH}{{\mathcal H}}
\newcommand{\CJ}{{\mathcal J}}
\newcommand{\CK}{{\mathcal K}}
\newcommand{\CL}{{\mathcal L}}
\newcommand{\CO}{{\mathcal O}}

\newcommand{\CS}{{\mathcal S}}
\newcommand{\CT}{{\mathcal T}}
\newcommand{\CU}{{\mathcal U}}

\renewcommand{\vec}{\bm}
\newcommand{\Cone}{\operatorname{Cone}}
\newcommand{\ant}{\operatorname{ant}}

\newcommand{\Lie}{\operatorname{Lie}}
\newcommand{\Id}{\operatorname{Id}}
\newcommand{\GL}{\operatorname{GL}}
\newcommand{\Gr}{\operatorname{Gr}}
\newcommand{\lleft}{\operatorname{left}}
\newcommand{\rright}{\operatorname{right}}
\newcommand{\U}{\operatorname{U}}

\newcommand{\hol}{\operatorname{hol}}
\newcommand{\Hom}{\operatorname{Hom}}
\newcommand{\Sym}{\operatorname{Sym}}
\newcommand{\Int}{\operatorname{Horn}}
\newcommand{\Horn}{\operatorname{Horn}}
\newcommand{\tr}{\operatorname{tr}}
\newcommand{\edim}{\operatorname{edim}}
\newcommand{\rank}{\operatorname{rank}}
\newcommand{\full}{\operatorname{full}}

\renewcommand{\c}{{\mathfrak{c}}}
\newcommand{\g}{{\mathfrak{g}}}

\newcommand{\h}{{\mathfrak{h}}}

\renewcommand{\k}{{\mathfrak{k}}}

\newcommand{\p}{{\mathfrak{p}}}
\newcommand{\gl}{{\mathfrak{gl}}}
\renewcommand{\t}{{\mathfrak{t}}}
\renewcommand{\u}{{\mathfrak{u}}}

\renewcommand{\vec}{\bm}
\usepackage{tikz}
\usepackage{bm}
\usepackage{mathtools}

\makeindex             

\begin{document}


\title*{ Tensor product of holomorphic discrete series representations of $U(p,q)$ and quivers }
 \titlerunning{$\Horn(p,q)$}
\author{Velleda Baldoni and Mich\`ele Vergne}

\institute{Velleda Baldoni \at  Dipartimento di Matematica, Universita degli
studi di Roma Tor Vergata, Via della Ricerca Scientifica  1-00133 Roma,
Italy, \email{ baldoni@mat.uniroma2.it}
\and Mich\`ele Vergne \at Institut de Math\'ematiques de Jussieu-Paris Rive Gauche, Universit\'e Paris-Cit\'e,  Campus Sophie Germain, Paris 13, France,
\email{michele.vergne@imj-prg.fr}}
%
\maketitle

\begin{abstract}{Consider the convex cone of holomorphic orbits in the Lie algebra of $U(p,q)$.
As in the classical case of Hermitian matrices, the set of elliptic orbits obtained contained  in the  sum of two holomorphic orbits
can be described  by ``Horn inequalities".  Using representations of quivers, we give another proof  of these Horn recursion inequalities  obtained by P-E Paradan.
 We recall the implications of these inequalities for decomposition of tensor product of two representations of the holomorphic discrete series of the group $U(p,q)$.}
 \end{abstract}

\section{Introduction}
\label{Intro}

Let $(\lambda,\mu)$ be a couple of dominant weights for the compact group $G=U(n)$.
One wants to describe the irreducible representations $V_\nu$ of $G$ occurring
with non zero multiplicity $c_{\lambda,\mu}^{\nu}$ in $V_\lambda\otimes V_\mu$. Here $V_\lambda,V_\mu,V_\nu$ are the irreducible representations of $U(n)$ (or $GL(n)$)
with highest weights $\lambda,\mu,\nu$ respectively.

Let  $\Lie(G)$ be the Lie algebra of $G$. Then the space $\sqrt{-1}\Lie(G)$ is the space of $n\times n$ Hermitian matrices.
Let $T$ be the  Cartan subgroup of $G$ consisting of $n\times n$ diagonal matrices with entries of modulus $1$.
 If
$a=(a(1),\ldots, a(n))$ is  a sequence of $n$ real numbers.
we denote also by $a\in \sqrt{-1}\Lie(T)$ the diagonal $n\times n$  Hermitian matrix with coefficients $a(i)$ on the diagonal.
 Using the inner product $ \tr(AB)$ on  matrices,
we may consider $a$ as an element of $(\sqrt{-1}\Lie(G))^*$.
Then  the Weyl chamber $\c_{\geq 0}=\{a=(a(1)\geq a(2)\geq \cdots \geq a(n))\}$ is a set of representatives of the coadjoint orbits of $G$ in $(\sqrt{-1}\Lie(G))^*$.
Denote by $O_a$ the  coadjoint orbit of $a$.
The orbit $O_a$ is provided with
a $G$-invariant K{\"a}hler structure.

Let $\lambda=(\lambda(1)\geq \lambda(2)\geq \cdots \geq \lambda(n))$ be  a dominant weight, that is an element of $\c_{\geq 0}$ with integral coefficients.
Since the representation $V_\lambda$ of $U(n)$  is  obtained by geometric quantization of the corresponding coadjoint orbit $O_\lambda$,   a geometric analogue of the  property  $V_\nu \subset V_\lambda\otimes V_\mu$  for $(\lambda,\mu,\nu)$ a triple of dominant weights
is the property
  $O_c \subset O_a+O_b$ for   $(a,b,c)$  a triple  of $n\times n$  Hermitian matrices.
 Let $\Cone(n)\subset \c_{\geq 0}\oplus \c_{\geq 0}\oplus \c_{\geq 0}$ be the set of such triples $(a,b,c)$.
The cone  $\Cone(n)$ is a polyhedral cone, and A. Horn (\cite{MR0140521}) conjectured a series of recursive inequalities  describing this cone.
Knutson-Tao (\cite{MR1671451})
 proved Horn conjecture together with the famous saturation conjecture: one has $c_{\lambda,\mu}^{\nu}\neq 0$ if and only if
there exists an integer $N\geq 1$ such that
  $c^{N\nu}_{N\lambda,N\mu}\neq 0$.
Equivalently  $c_{\lambda,\mu}^{\nu}\neq 0$ if and only if $(\lambda,\mu,\nu)\in \Cone(n)$.

Consider now the non compact group $G=U(p,q)\subset GL(p+q)$ of pseudo-unitary transformations, and let $n=p+q$.
The group $G$ has a compact Cartan subgroup $T$: the space of $n\times n$  diagonal matrices with entries of modulus one.
There is a series $\pi_\lambda$ of unitary irreducible representations of $G$
called the holomorphic discrete series associated to  a subset of characters of $T$.
Such a   representation  $\pi_\lambda$ is similarly indexed by a sequence of slowly decreasing  integers
$$\lambda=(\lambda(1)\geq\cdots\geq\lambda(p)\geq  \lambda(p+1)\geq\cdots\geq\lambda(p+q)),$$ with a gap between $\lambda(p)$ and $\lambda(p+1)$ (see Section \ref{ds}, Ex.\ref{exUpq}).
It is easy to prove that the tensor product of two representations  $\pi_\lambda$ and $\pi_\mu$ of the holomorphic discrete series decomposes as a sum with finite multiplicities of   representations
$\pi_\nu$ belonging to the holomorphic discrete series:
\begin{equation}\label{tensorproduct}\pi_\lambda \otimes \pi_{\mu}=\oplus_\nu m_{\hol}^{\nu}(\lambda,\mu )\pi_\nu.\end{equation}
Thus a natural question is to describe the triples $(\lambda,\mu,\nu)$ such that
$\pi_{\nu}\subset \pi_{\lambda}\otimes \pi_{\mu}$.

 Let $$A=(A(1)\geq\cdots\geq A(p)\geq A(p+1)\geq\cdots\geq A(p+q))$$ be  a sequence  of slowly decreasing real numbers. We also denote by $A\in \sqrt{-1}\Lie(T)$
the diagonal  matrix with coefficients $A(i)$ on the diagonal. Using again  the bilinear form  $\tr(AB)$, we may consider $A$  as an element of
 $(\sqrt{-1}\Lie(G))^*$.
Let $\CO^{nc}_A$ be the corresponding (non compact) coadjoint orbit of $A$  by the action of $U(p,q)$. Then $\CO^{nc}_A$ is provided with
a $U(p,q)$-invariant K{\"a}hler structure and we say that $\CO^{nc}_A$ is an holomorphic orbit.
The  representation $\pi_\lambda$ of $U(p,q)$ of the holomorphic discrete series  is obtained by geometric quantization of the corresponding (non compact) coadjoint holomorphic orbit $\CO^{nc}_\lambda$.
So the  geometric analogue of   the property $\pi_{\nu}\subset \pi_{\lambda}\otimes \pi_{\mu}$ is  the property  $\CO^{nc}_C \subset \CO^{nc}_A+\CO^{nc}_B$.
 Let $\Cone_{\hol}(p,q)$  be the set of such triples $(A,B,C)$.

In (\cite{Hornpq}), P-E Paradan described inductive necessary and sufficient conditions
for $(A,B,C)$ to belong to $\Cone_{\hol}(p,q)$, similar to the Horn conditions in the definite case.
 Using Derksen-Weyman saturation theorem (\cite{MR1758750}), one deduces
that the saturation conjecture holds for discrete series representations: there exists an integer $N\geq 1$ such that
$m_{\hol}^{N\nu}(N\lambda,N\mu)\neq 0$ if and only if
$m_{\hol}^{\nu}(\lambda,\mu)\neq 0$.
So inequalities  for the cone $\Cone_{\hol}(p,q)$ give inductive necessary and sufficient conditions
for $m_{\hol}^{\nu}(\lambda,\mu)$ to be non zero.

 In this article, we give another proof of Paradan  conditions by relating $m_{\hol}^{\nu}(\lambda,\mu)$ to quivers.
One benefit   is that the inductive conditions appear very naturally in this setting.  Furthermore, some  symmetries are visible which are not seen in the initial problem (see Proposition \ref{symetry}).
We also sketch another proof (based on Rossmann formula for Fourier transform of orbits (\cite{BV82})) that the cone generated by the triples
$(\lambda,\mu,\nu)$ with  $m_{\hol}^{\nu}(\lambda,\mu)>0 $ is equal to the cone $\Cone_{\hol}(p,q)$, as it should be.

Recall  the general quiver setting.
Let $Q=(Q_0,Q_1)$ be a quiver, where~$Q_0$ is the finite set of vertices and~$Q_1$ the finite set of arrows.
A \emph{dimension vector} for~$Q$ is a vector~${\bf {n}} = (n_x)_{x\in Q_0}$ of nonnegative integers
(if vertices are labeled  $\{1,2,\ldots,  q_0\}$,
we  adopt also the notation ${\bf {n}} = [n_{1}, \ldots,  n_{q_0}]$).
We consider the family of complex vector spaces~$\CE=(E_x)_{x\in Q_0}$ with  $E_x=\C^{n_x}.$
The space of \emph{representations} of the quiver~$Q$ on~$\CE$ is given by
\[ \CH_Q(\CE) \coloneqq \bigoplus_{a:x\to y\in Q_1} \Hom(E_x, E_y). \]
The Lie group~$\GL_Q({\bf{n}}) = \prod_{x\in Q_0} \GL(n_x)$  acts naturally on $\CH_Q(\CE)$, and so acts naturally on the space $\Sym^*(\CH_Q(\CE))$ of polynomial functions on~$\CH_Q(\CE)$.
Decompose
\[ \Sym^*(\CH_Q(\CE))=\bigoplus_{\bm{\lambda}} m_Q(\bm{\lambda}) V_{\bm{\lambda}}, \]
where $V_{\vec\lambda}$ denotes the irreducible representation of $\GL_Q({\bf n})$ with highest weight $\bm\lambda$.
If $Q$ has no cycles, the multiplicity  $m_Q(\bm{\lambda})$ is finite.

When $Q$ is the  Horn quiver  $\CH_2$
\begin{equation}\label{eq:quiver}
  1 \rightarrow 3 \leftarrow 2
\end{equation}
and the dimension vector is $\bm n=[n,n,n]$, then $m_Q([\lambda,\mu,\nu])=c_{\lambda, \mu}^{\nu^*}$ where $\nu^*=(-\nu(n)\geq\cdots \geq -\nu(1))$
indexes the dual representation $V_\nu^*$.

One can decide  if $m_Q(\bm{\lambda})>0$ by studying the Schubert positions of general subrepresentations of a representation of $Q$ on $\CE$. One obtains a natural inductive criterium (\cite{bvw}) which coincides for the quiver $\CH_2$
 with Belkale inductive criterium (\cite{MR2177198}) for intersection of Schubert cells.

The action of the compact group $K=\U_Q({\bf{n}}) = \prod_{x\in Q_0} \U(n_x)$ on  $\CH_Q(\CE)$ leads to  a moment map  with
 values in $(\sqrt{-1}\Lie(K))^*$ and to a ``moment cone" $\Cone_Q(\CE)$   parameterizing coadjoint  orbits of $K$  contained in the image of the moment map.
The saturation theorem of Derksen-Weyman is equivalent to the statement that $m_Q(\bm{\lambda})>0$ if and only if $\bm{\lambda}$ belongs to the moment cone  $\Cone_Q(\CE)$.

\bigskip

Consider  the quiver $Q=Q_{3,3}$ with $6$ vertices \{$1$, $2$, $3$, $4$, $5$, $6$\}  and arrows $1\to 3$, $2\to 3$, $4\to 3$, $4\to 5$, $4\to 6,$ see Fig: \ref{fig-vell1}.
Families of objects parameterized by $Q_0$ will be written as lists.

\begin{figure}[h]
\hbox to \hsize\bgroup\hss
\beginpicture
\setcoordinatesystem units <1.2in, 0.9in>

\setplotarea x from -2 to 2, y from -.8 to .8
\multiput {\Large$\bullet$} at -1.22 .72  -1.22 -.72  -.5 0  .5 0  1.22 .72  1.22 -.72 /
\setplotsymbol ({\rm .})
\plot -1.22 .72  -.5 0  .5 0  1.22 .72 /
\plot -1.22 -.72  -.5 0 /
\plot  .5 0  1.22 -.72 /
\put{\Large $1$} at -1.28 .88
\put{\Large $2$} at -1.28 -.88
\put{\Large $3$} at -.45 .13
\put{\Large $4$} at .45 .13
\put{\Large $5$} at 1.28 .88
\put{\Large $6$} at 1.28 -.88
\arrow <12pt> [.3,.7] from -.4 0 to -.47 0
\arrow <12pt> [.3,.7] from 1.148 .648 to 1.2 .705
\arrow <12pt> [.3,.7] from 1.148 -.648 to 1.2 -.705
\arrow <12pt> [.3,.7] from -.572 -.072 to -.53 -0.03
\arrow <12pt> [.3,.7] from -.572 .072 to -.53 0.03
\endpicture
\hss\egroup \caption{\label{fig-vell1}}
\end{figure}

\bigskip

Consider the dimension vector   ${\bm \beta}=[p,p,p,q,q,q]$ and  let  $$\CE=[\C^p,\C^p,\C^p,\C^q,\C^q,\C^q].$$
We denote  the moment cone $\Cone_Q(\CE)$  by $\Cone_Q(p,q)$. It is contained in ${\bm \c}_{\geq 0}=\oplus_{i=1}^6 \c^{(i)}_{\geq 0}$,
where  $\c^{(i)}_{\geq 0}$ is  the positive Weyl chamber for the compact group  $\U(\beta_i)$.

\bigskip

{\bf Example 1; $p=2$, $q=1$}\label{u21intro}.

{\em   Let ${\bm a}=[a_1,a_2,a_3,a_4,a_5,a_6]\in {\bm \c}_{\geq 0}$:
   \begin{align*}
    a_1&=(a_1(1)\geq a_1(2)), \\
    a_2&=(a_2(1)\geq a_2(2)), \\
     a_3&=(a_3(1)\geq a_3(2)),\\
    a_4&=(a_4(1)), \\
 a_5&= (a_5(1)),\\
  a_6&=( a_6(1)).
  \end{align*}

Then ${\bm a}$ belongs to $\Cone_Q(2,1)$ if  and only if
\begin{align*}
  a_1(1)+a_1(2)+a_2(1)+a_2(2)+a_3(1)+a_3(2)+a_4(1)+a_5(1)+a_6(1)&=0,\\
   a_1(2)\geq 0,\,\, a_2(2)\geq 0,\,\, a_5(1)\leq 0,\,\, a_6(1)&\leq 0, \\
   a_1(2)+a_2(2)+a_3(2)+a_4(1)+a_5(1)+a_6(1)&\leq 0,\\
a_1(1)+a_1(2)+a_2(1)+a_2(2)+a_3(1)+a_3(2)&\leq 0,\\
    a_1(1) + a_2(2) +a_3(2) &\leq 0, \\
    a_1(2) +a_2(1) + a_3(2) &\leq 0, \\
    a_1(2) + a_2(2) + a_3(1) &\leq 0.\\
    \end{align*}}


\bigskip

Let  ${\bm \lambda }=[\lambda_1,\lambda_2,\lambda_3,\lambda_4,\lambda_5,\lambda_6]$ be a dominant weight for
$U_Q({\bm \beta})$.
If $m_Q({\bm \lambda })>0$, it follows that
$\lambda_1,\lambda_2,\lambda_3^*$ indexes polynomial representations of $GL(p)$, $\lambda_4^*,\lambda_5,\lambda_6$ dual of polynomial representations of $Gl(q)$. We can consider the concatenated sequences
  $A=(\lambda_1,\lambda_5)$, $B=(\lambda_2,\lambda_6)$, $C=(\lambda_3^*,\lambda_4^*)$.
They are slowly decreasing sequences of $n$ integers.
Then the main remark of this article is that

\begin{proposition}\label{parqui}
$$m_Q({\bm \lambda})=m_{\hol}^C(A,B).$$
\end{proposition}

Thus we can read on the equations of the cone $\Cone_Q(\CE)$ when $m_{\hol}^C(A,B)>0$, or equivalently when
 $\CO^{nc}_C\subset \CO^{nc}_A+\CO^{nc}_B$.

\bigskip
{{\bf Continuation of Example 1}

{\em  Let
\begin{align*}
    A&=(A(1)\geq A(2) \geq A(3)), \\
    B&=(B(1)\geq B(2) \geq A(3)), \\
    C&=(C(1)\geq C(2)\geq C(3)),\\
    \end{align*}
 be a triple of slowly decreasing sequence of real numbers indexing holomorphic orbits in $U(2,1)$. Using Proposition \ref{parqui},
we reobtain (see \cite{Hornpq})
 $$\CO^{nc}_C\subset \CO^{nc}_A+\CO^{nc}_B$$ if and only if
\begin{align*}
A (1)+A (2) +A(3) +B(1)+B(2)+B(3) =C(1)+C(2)+C(3)&,\\
A(1)+B(1)\geq C(2) &,\\
A(1)+A(2)+B(1)+B(2)\leq C(1)+C(2)&,\\
A(2)+B(2)\leq C(2) &,\\
A(2)+B(1)\leq C(1)&,\\
A(1)+B(2) \leq C(1)&.\\
\end{align*}}

\bigskip

Here is the outline of this article.

In Section  \ref{sec1}, we recall our results (\cite{bvw}) on general subrepresentations of representations of quivers
 and the general inductive algorithm to determine inequalities of moment cones  $\Cone_Q(\CE)$.

 In Section \ref{algo}, we explain in more details the algorithm, and some of its simplifications in the particular situation of the quiver $Q_{3,3}$. We point out the existence of symmetries in the problem.

In Section \ref{ds}, we recall the definition of the holomorphic discrete series for a general Hermitian symmetric space $G/K$.
We recall that the  multiplicity of  a representation $\pi_\nu$ in the tensor product of two  representations  $\pi_\lambda$, $\pi_\mu$
of the holomorphic discrete series
 can be computed as   a multiplicity  for the action of the  maximal compact
subgroup $K$ of $G$  in an explicit  representation.

In Section  \ref{orbit}, we  state the geometric  analogue for the description of  the sum of  two holomorphic orbits.
 As obtained by P.E. Paradan, the set of  orbits contained in the sum of two holomorphic orbits $\CO^{nc}_A$, $\CO^{nc}_B$ can be described as a moment cone for an action of $K$ on a (non compact)
Hamiltonian space. We sketch a more direct proof based on Rossmann formula for Fourier transforms of orbits.

In Section \ref{quiverrep}, we consider $G=U(p,q)$ and relate multiplicities $m_{\hol}^{\nu}(\lambda,\mu)$ to
 $m_Q({\bm \lambda})$, where $Q=Q_{3,3}$ is the particular quiver described above. This is the main observation of this article.

%
%
In  Section  \ref{coneeq},
we compare our inequalities with the ones obtained by P.E. Paradan.

In Section \ref{examples}, we give some more examples of our inequalities.

%
%
%
%
%
%

\bigskip

The content of this article was presented in the conference " Symmetry in geometry and analysis" (Reims, June 2022).
 We  hope that  this contribution to the problem of branching rules  is an appropriate way   of  honoring
  Toshiyuki KOBAYASHI.

As we pointed out, results  of this article are due to P.E. Paradan. But we wished to show that quivers  can be
useful for some special problems in representations of non compact groups.
In particular, at least in low dimensions, it is easy to  guess and visualize the corresponding "Horn inequalities".

  We thank Jean-Louis Clerc, and  David Vogan for their comments  during the conference.
We thank Michel Duflo for many  comments and references.

\section*{ Notations and conventions}
All vector spaces will be fi\-nite-di\-men\-si\-o\-nal complex vector spaces.
Given a vector space~$E$, we write $\dim E$ for its (complex) dimension and we denote by $\Gr(r,E)$ the Grassmannian of subspaces of dimension~$r$ of~$E$, where $0\leq r\leq\dim E$.
We use calligraphic or  bold letters to denote families of objects labeled by the vertex set~$Q_0$ of a quiver.
For example, $\CE=(E_x)_{x\in Q_0}$ will be a family of vector spaces indexed by~$Q_0$, ~${\bf n}=(n_x)_{x\in Q_0}$ will be a family of natural numbers.
If $Q_0$ is labeled  as a set  $\{1,2,\ldots, q_0\}$ of consecutive integers, families of
objects labeled by the vertex set~$Q_0$ will be written as lists.
For example, the dimension vector $[3,1,2,1,1,1]$ for the  quiver $Q_{3,3}$ means ${\bf n}=(n_x)$ with $n_1=3$, $n_2=1$, $n_3=2$, $n_4=1,n_5=1,n_6=1$.

\section{Space of representations of a quiver $Q$}
\label{sec1}

Let $Q=(Q_0,Q_1)$ be a quiver, where~$Q_0$ is the finite set of vertices and~$Q_1$ the finite set of arrows.
We use the notation~$a:x\to y$ for an arrow~$a\in Q_1$ from~$x\in Q_0$ to~$y\in Q_0$.
We allow $Q$ to have  multiple arrows between two vertices but no cycles.
Let ${\bf {n}} = (n_x)_{x\in Q_0}$  be a dimension vector.
We consider a family of complex vector spaces~$\CE=(E_x)_{x\in Q_0}$ of dimension vector $\bf{n}$ that is $\dim E_x=n_x.$
Conversely, a family  $ \CE=(E_x)_{x\in Q_0}$ of vector spaces indexed by $Q_0$ defines the dimension vector $\dim \CE=(\dim E_x)_{x\in Q_0}$.
So  the expression dimension vector will refer indifferently  to a family of nonnegative integers or to a family of complex vector spaces.

The space of \emph{representations} of the quiver~$Q$ on~$\CE$ is given by
\begin{align}\label{eq:quiver repr}
  \CH_Q(\CE) \coloneqq \bigoplus_{a:x\to y\in Q_1} \Hom(E_x, E_y),
\end{align}
whose elements are families~$r=(r_a)_{a\in Q_1}$ of linear maps~$r_a: E_x\to E_y$, one for each arrow $a:x\to y$ in $Q_1$.
The Lie group~$\GL_Q(\CE) = \prod_{x\in Q_0} \GL(E_x)$  acts naturally on $\CH_Q(\CE)$.
For  $g\in\GL_Q(\CE)$,  and $r\in\CH_Q(\CE)$, the action is $ (g \cdot r)_{a:x\to y\in Q_1} = g_y r_a g_x^{-1}$.

Thus the group $\GL_Q(\CE)$ acts naturally on the space $\Sym^*(\CH_Q(\CE))$ of polynomial functions on~$\CH_Q(\CE)$.
  Describing the  irreducible representations of $GL_Q(\CE)$ occurring with non zero multiplicities in $\Sym^*(\CH_Q(\CE))$
  is  related to the study of subrepresentations of a representation $r$. Let us recall some of the relevant results.

Let $\bm{\alpha},\bm{\beta}$ be two dimension vectors.
Consider the Ringel bilinear form  $$
  \langle{\bm\alpha},{\bm\beta}\rangle= \sum_{x\in Q_0} \alpha_x \beta_x - \!\!\!\!\sum_{a:x\to y\in Q_1}\!\!\!\! \alpha_x\beta_y.$$

Let ${\bm {\alpha}}=(\alpha_x)_{x\in Q_0}$ be a dimension vector  such that  $\alpha_x\leq \dim E_x$.
We say that  $\bm{\alpha}$ is a \emph{subdimension vector} for $\CE.$

Define the expected dimension $\edim_Q$ by the formula:
 \begin{equation} \edim_Q(\bm {\alpha},\CE)=\langle \bm {\alpha},\bm{\beta} \rangle \end{equation} \label{edim}
where $\bm{\beta}$ is the dimension vector defined by $\beta_x = \dim E_x - \alpha_x.$

We write $\CS\subseteq\CE$ if $\CS=(S_x)_{x\in Q_0}$ is a family of subspaces~$S_x \subseteq E_x$.
The family~$\CS$ is called a \emph{subrepresentation} of~$r\in\CH_Q(\CE)$ if $r_a S_x\subseteq S_y$ for every arrow $a: x\to y$ in~$Q_1$;
we abbreviate this condition by $r \CS \subseteq \CS$.

Schofield~(\cite{MR1162487}) characterized (inductively) the subdimension vectors~$ {\bm \alpha}$ such that any~$r\in\CH_Q(\CE)$ has a subrepresentation~$\CS$ with $\dim\CS =  {\bm \alpha}$.
We call such a dimension vector a \emph{Schofield subdimension vector} for~$\CE$ and denote this property by $ {\bm \alpha}\leq_Q  {\bm n}$, where $\dim\CE = {\bm n}$.
We also write $ {\bm \alpha} <_Q  {\bm n}$ if in addition at least one of the inequalities $\alpha_x \leq n_x$ is strict.
The condition $ {\bm \alpha} \leq_Q {\bm \beta}$ is transitive.
It is easy to see that a necessary condition for ${\bm \alpha}$ to be a \emph{Schofield subdimension vector} for $\CE$ is that $\edim_Q({\bm \alpha},\CE)\geq 0$.
In \cite{bvw}, we  proved the following natural criterium, which refines Schofield criterium:

\begin{theorem}

Let ${\bm \alpha}$ be a subdimension vector for $\CE$.
Then ${\bm \alpha}\leq_Q \dim \CE$ if and only if

\begin{enumerate}
\item $\edim_Q({\bm \alpha},\CE)\geq 0$.

 \item If ${\bm \beta}<_Q {\bm \alpha}$, then ${\bm \beta}<_Q {\dim \CE }$  (transitivity).

  \end{enumerate}

  \end{theorem}

Consider
\begin{align*}
  \Gr_Q(\bm{\alpha},\CE) \coloneqq \prod_{x\in Q_0} \Gr_Q(\alpha_x,E_x)
\end{align*}
where $\Gr(\alpha_x,E_x)$ denotes the Grassmannian of subspaces of~$E_x$ of dimension~$\alpha_x$.
The dimension of $\Gr_Q(\bm{\alpha},\CE)$ is $\sum_{x\in Q_0} \alpha_x \beta_x$, where $\bm{\beta}$
is the dimension vector defined by $\beta_x = \dim E_x - \alpha_x.$
Thus, we have
\begin{equation}\label{edim}
\edim_Q({\bm \alpha},\CE)=\dim(\Gr_Q(\bm{\alpha},\CE))- \!\!\!\!\sum_{a:x\to y\in Q_1}\!\!\!\! \alpha_x\beta_y.
\end{equation}

Given a representation~$r\in\CH_Q(\CE)$ and a dimension vector~$\bm\alpha$, we define the corresponding \emph{quiver Grassmannian} by
\begin{align*}
  \Gr_Q(\bm{\alpha},\CE)_r \coloneqq \{ \CS \in \Gr_Q(\bm{\alpha},\CE) : r\CS \subseteq \CS \}.
\end{align*}

In this language, a Schofield subdimension vector is a subdimension vector~$\bm\alpha$ such that $\Gr_Q(\bm{\alpha},\CE)_r \neq\emptyset$
for every representation~$r\in\CH_Q(\CE)$.
 If $\alpha$ is a Schofield   subdimension vector, then the "expected dimension" $\edim_Q(\bm {\alpha},\CE)$ is the dimension of
 the variety $\Gr_Q(\bm{\alpha},\CE)_r$ for generic $r$.

Choose now a Borel subgroup $B$ of $GL_Q(\CE)$ (or a complete filtration of $\CE$). If $\CS\subseteq \CE$, the Borel subgroup of $\CE$ determines a Borel subgroup of $GL_Q(\CS)$, still denoted by $B$.
By definition, a Schubert variety ${\bm \Omega}$ is  the closure of the $B$ orbit of an element  $\CS$ in $\Gr_Q(\bm{\alpha},\CE)$
(thus ${\bm \Omega}=(\Omega_x)_{x\in Q_0}$ is the product of Schubert varieties  $\Omega_x$ in $\Gr_Q(\alpha_x,E_x )$).
Given a representation $r\in\CH_Q(\CE)$, and ${\bm \Omega}$ a Schubert variety, define
\begin{align*}
  \bm\Omega_r
\coloneqq \Gr_Q(\bm\alpha,\CE)_r \cap \bm\Omega
= \{ \CS \in \bm\Omega : r \CS \subseteq \CS \}.
\end{align*}

\begin{definition}
We say that  $\bm \Omega$ is $Q$-intersecting if $\bm \Omega_r$ is non empty for every  $r\in\CH_Q(\CE)$.
\end{definition}

In other words, $\bm \Omega$ is $Q$-intersecting if, for every $r\in\CH_Q(\CE)$, the Schubert variety~$\bm\Omega$ contains a subrepresentation of~$r$
(it is enough to consider generic representations $r$).
If $\bm\Omega$  is the closure of the $B$-orbit of  $\CS$, we write  $\CS \subseteq_{Q,B} \CE$ if $\Omega$ is $Q$-intersecting .
 We write $\CS \subset_{Q,B} \CE,$ if $\CS\neq \CE$.
Clearly, a necessary condition for $\bm\Omega$ to be $Q$-intersecting is that $\bm\alpha$ is a Schofield subdimension vector.

If $Q$ is the Horn quiver $\CH_2$
\begin{equation}
  1 \rightarrow 3 \leftarrow 2
\end{equation} and the  vector dimension $[n,n,n]$, a triple  $\bm \Omega=[\Omega_1,\Omega_2,\Omega_3]$  of Schubert varieties in the Grassmannian $\Gr(r,n)$ is $Q$-intersecting if and only if the Schubert varieties  $(\Omega_1,\Omega_2,\Omega_3)$ intersect in the homological sense.
So this explains our terminology.

Our inductive criterion for $\bm\Omega$ to be $Q$-intersecting is based on a numerical quantity: the expected dimension $\edim_{Q,B}(\bm\Omega,\CE)$  generalizing (\ref{edim}).
\label{eQB}\begin{align}
  \edim_{Q,B}(\bm\Omega,\CE) \coloneqq \dim\bm\Omega - \!\!\!\!\sum_{a\colon x\to y\in Q_1}\!\!\!\! \alpha_x \beta_y.
\end{align}

(So when ${\bm \Omega}=\Gr_Q(\bm\alpha,\CE)$ is the closure of the big cell,  $\edim_{Q,B}(\bm\Omega,\CE)=\edim_Q(\bm\alpha,\CE)$).

 If $\bm\Omega$  is the closure of the $B$-orbit of  $\CS$, we write   $\edim_{Q,B}(\CS,\CE)$ instead of  $\edim_{Q,B}(\bm\Omega,\CE)$.

If ${\bm \Omega}$ is $Q$-intersecting,  then it is easy to see that the expected dimension  $\edim_{Q,B}(\bm\Omega,\CE)$ is the dimension of the variety ${\bm \Omega}_r$ for generic $r$.
So a necessary condition for $\bm\Omega$ to be $Q$-intersecting in~$\CE$ is that $\edim_{Q,B}(\bm\Omega,\CE)\geq0$
 (but this condition is not sufficient).

Our main result (\cite{quivershort}, \cite{bvw}) is the following inductive criterium:

\begin{theorem}\label{thm:main}
Let $\CE$ be a family of vector spaces,  and $\CS$ a family of subspaces of~$\CE$.
Then $\CS \subseteq_{Q,B} \CE$ if and only if
\begin{enumerate}
\item\label{it:main A} $\edim_{Q,B}(\CS,\CE) \geq 0$,
\item\label{it:main B} $\CT \subset_{Q,B} \CE$ for every $\CT \subset_{Q,B} \CS$(transitivity).
\end{enumerate}
\end{theorem}

Let us recall the consequences of this theorem for multiplicities.

For~$x\in Q_0$, let~$E_x=\C^{n_x}$, with standard basis~$(e_j)_{1\leq j\leq n_x}$, and consider
the Borel subgroup~$B_x$ that consists of the upper-triangular matrices in~$\GL(E_x)=GL(n_x)$.
Let $U_x$ be the maximally compact subgroup of $GL(n_x)$ consisting of unitary operators, with Lie algebra~$\u_x$.
We  identify  $(\sqrt{-1}\u_x)^*$  with the space of Hermitian operators on~$E_x$ as in the introduction.
Consider the Weyl chamber~$C_x\subset (\sqrt{-1} \u_x)^*$  which may be identified to  the cone of slowly decreasing sequences $a$ of $n_x$
  real numbers  $a=(a_x(1)\geq\cdots\geq a_x(n_x))$.
A $\Z$-valued element of $C_x$ is a dominant weight: $\lambda_x=(\lambda_x(1)\geq\cdots\geq \lambda_x(n_x))$.

The irreducible representations of $\GL_Q(\mathcal E)$ are of the form $V_{\bm\lambda} = \bigotimes_{x\in Q_0} V_{\lambda_x}$,
where $\bm\lambda = (\lambda_x)_{x\in Q_0}$; here $\lambda_x$ is a dominant weight, and   $V_{\lambda_x}$ is the irreducible representation of
$GL(E_x)$ with   highest weight $\lambda_x$.

Write $$\Sym^*(\CH_Q(\CE))=\bigoplus_{\bm{\lambda}} m_Q(\bm{\lambda}) V_{\bm{\lambda}}.$$

Consider the one parameter subgroup $(e^{i\theta} {\rm Id}_{E_x})_x$.  It acts trivially in $\CH_Q(\CE)$.
Thus for  $m_Q(\bm{\lambda})$ to be positive, we need
 $$\sum\limits_{\crampedclap{x\in Q_0}} \sum_{j=1}^{n_x} \lambda_x(j) = 0.$$
 We will refer to this equation as the center equation.

By definition, the cone $\Cone_Q(\CE)$ is the cone (in  the Weyl chamber $\oplus_{x\in Q_0} C_x$)  generated by the dominant weights $\bm\lambda$ with
$m(\bm{\lambda})>0$.
It is easy (as pointed out to us by Ressayre, \cite{ressayrepc}) to deduce from Derksen-Weyman \cite{MR1758750} the saturation property for the coefficients $m_Q(\bm{\lambda})$, namely
 $m_Q(\bm{\lambda})>0 $ if and only if there exists a positive integer $N$ with  $m_Q(N\bm{\lambda})>0$.

The cone $\Cone_Q(\CE)$ has an alternate description  in terms of a moment map in the sense of symplectic geometry.
The moment map for the action of the maximally compact subgroup $U(\CE) = \prod_{x\in Q_0} U_x$ on $\CH_Q(\CE)$   with value in
$\oplus_x (\sqrt{-1}\u_x)^*$ is given by
\begin{align*}
  \mu\colon \CH_Q(\CE) \to \bigoplus_x (\sqrt{-1}\u_x)^*, \quad
  r=(r_a)_{a\in Q_1} \mapsto \mu(r) = (\mu_x(r))_{x\in Q_0},
\end{align*}
where~$\mu_x(r)$ is the Hermitian matrix $\sum_{y,b:y\to x} r_b r_b^* - \sum_{y,a : x\to y} r_a^* r_a$.

\bigskip

If $n$ is an integer, we denote by $[n]$ the list $[1,2,\ldots,n]$.
If $K$ is a subset of $[n]$ with  $r$ elements, we write $K\subseteq [n]$
and also
$K=\{K(1), K(2)\,\ldots,K(r)\}$ the corresponding  sublist of $[n]$ indexed as an increasing list of integers.

 Let $|K|=r$ be the cardinal of the set $K$.
If $J\subseteq [|K|]$ is a subset of $[1,2,\ldots,r]$ with $s$ elements,
 we can compose $K$ and $J$
and obtain the  subset $\{K(J(1)),\ldots, K(J(s))\}$  of $[n]$ with $s$ elements.

Any subset $K \subseteq [n]$  of cardinal $r$ determines the subspace
 $S = \oplus_{i \in K} e_i$  (where $e_i$ denotes the standard basis of~$\C^n$) of $\C^n$  and hence a Schubert variety~$\Omega$ in $\Gr(r,n)$.
 The empty subset $K=\{\}$  indexes the subspace $S=\{0\}$ of $\C^n$.
The  subset $K=\{1\}$ indexes  $\Omega=\{\C e_1\}$, that is the $B$ fixed  point in $\Gr(1,n)$.
The  subset $K=\{n\}$ indexes $\Omega=\Gr(1,n)$, etc....

We now consider families of such sets.
If ${\bm n}=(n_x)_{x\in Q_0}$ is a dimension vector, we denote by $[\bm n]=([n_x])_{x\in Q_0}$ the family  of the lists $[1,\ldots, n_x]$.
Any family $\CK \subseteq [ {\bm n}]$, by which we mean that~$\CK=(K_x)_{x\in Q_0}$
consists of subsets $K_x \subseteq \{1,\dots,n_x\}$, determines the  family~$\CS=(S_x)_{x\in Q_0}$
of subspaces $S_x = \oplus_{i \in K_x} e_i$, and hence a Schubert variety~$\boldsymbol\Omega$.
Any Schubert variety is indexed by a family $\CK$.
Let us write $\CK \subseteq_{Q,B} [ {\bm n}]$ to denote that $\CS \subseteq_{Q,B} \CE$, where $\CS$ is the family of  subspaces determined by $\CK$.
We denote by $\edim_{Q,B}(\CK,\CE)$ the numerical quantity $\edim_{Q,B}(\CS,\CE)$, where $\CS$ is the subspace indexed by $\CK$.

We now recall:

\begin{theorem}\label{thm:moment cone and rep theory summary}
For any dominant  weight $\bm\lambda = (\lambda_x)_{x\in Q_0}$ of $\GL_Q(\CE)$, the following are equivalent:
\begin{enumerate}
\item\label{cond:moment cone} $-\bm\lambda$ is in the image of the moment map,
\item\label{cond:highest weight cone} $\bm\lambda \in \Cone_Q(\CE)$,
\item $V_{\bm\lambda} \subseteq \Sym^*(\CH_Q(\CE))$,
\item\label{cond:horn} $\sum_{{x\in Q_0}} \sum_{j=1}^{n_x} \lambda_x(j) = 0$
and $\sum_{{x\in Q_0}} \sum_{j \in K_x} \lambda_x(j) \leq 0$ for all $\CK \subseteq_{Q,B} [\bf n]$.
\end{enumerate}
\end{theorem}

The equivalence of (1) and (2) is due to Mumford (\cite{NessMumford84}) and Guillemin-Sternberg (\cite{GS1982convex},\cite{GS1982qr}).
 The equivalence of (2) and (3) follows  from the saturation properties of the coefficients  $m_Q(\bm{\lambda})$.
 The equivalence of (3) and (4) follows from Ressayre's conditions for general GIT, the explicit version for this case being discussed
 in \cite{quivershort}, \cite{VW}.
In the inequalities (4), it is sufficient to consider the families $\CK$
such that $\edim_{Q,B}(\CK,\CE)=0$. Even with this restriction, one usually obtains a redundant set of inequalities for the description of the polyhedral cone $\Cone_Q(\CE)$.

\begin{definition}\label{def:Hornset}
We denote by  $\Int(Q,\CE)$  the set  of $\CJ$ such that  $\CJ \subseteq_{Q,B} [\bf n]$ and by
$\Int_0(Q,\CE)$ the subset of those $\CJ$ with $\edim_{Q,B}(\CJ,\CE)=0$.
 We call an element $\CJ$ in $\Int(Q,\CE)$ a $Q$-intersecting set.
 For $\CA$ a subset of $Q_1$, denote by  $\Int(\CA,Q, \CE)$  the subset of $Q$-intersecting sets $\CJ$  such that $|J_x|=|J_y|$ when there is an arrow $a:x\to y\in \CA$,
 and by $\Int_0(\CA,Q, \CE)$   the subset of those $\CJ\in \Int(\CA,Q,\CE)$ with $\edim_{Q,B}(\CJ,\CE)=0$.

 \end{definition}

As a consequence of Theorem \ref{thm:moment cone and rep theory summary}, to describe explicitly the cone $\Cone_Q(\CE)$ by inequalities, we only need to determine the
 "Horn set"  $\Int_0(Q,\CE)$.
 Theorem \ref{thm:main}  translates immediately into  an inductive numerical criterion   that we  explicit in Section \ref{algo}.

\bigskip

Let us make some  remarks on the conditions ${\bm \Omega}\subseteq _{Q,B}\CE$.

\begin{remark}\label{subquiver}

Let ${\bf \Omega}=(\Omega_x)_{x\in Q_0}$  be a family of Schubert varieties.
Let $P\subset Q$ be  a subquiver,  and   ${\bf \Omega}_P=(\Omega_x)_{x\in P},\CE_P=(E_x)_{x\in P}$.
Clearly a necessary condition for  ${\bf \Omega}$ to be $Q$-intersecting in $\CE$ is that  ${\bf \Omega}_P$ is $P$-intersecting in $\CE_P$.
\end{remark}

\begin{remark}\label{minandmax}
Let us give  a few elements in $\Int_0(Q,\CE)$ which are always present.

$\bullet$ The full list $\CK_{\full}$ with $K_x=[n_x]$ for any $x\in Q_0$, with corresponding space $\CS=\CE$.

$\bullet$ Let $v$ be a maximal vertex. Let $K_x=\{\}$ for $x\neq v$, and $K_{v}=\{1\}$.
This indexes the family $\CS$ with  $S_x=\{0\}$, for $x\neq v$, while $S_{v}=\C e_1$ in $E_v$.
Clearly $\CS$ is a subrepresentation of any $r\in \CH_Q(\CE)$.
So  for any maximal vertex $v$,  $\CK_{\max}^{(v)}=(K_{x,x\neq v}=\{\} ,K_{v}=\{1\})$ is always in $\Int_0(Q,\CE)$.

 In particular for $Q=Q_{3,3}$ and dimension vector $\bf n$,
  $$\CK_{\max}^{(5)}=[\{\}, \{\} ,\{\}, \{\}, \{1\}, \{\}]$$
  and  $$\CK_{\max}^{(6)}=[\{\}, \{\} ,\{\}, \{\}, \{\}, \{1\}] $$ are $Q$-intersecting sets.

 $\bullet$ Let us consider now $u$ a minimal vertex.
Let
$K_{u}=\{1,2,\ldots, n_{u}-1\}$ while $K_x=\{1,2,\ldots, n_{x}\}$ if $x\neq u$. This indexes the family
$S_{u}=\C^{n_{u}-1}$ (a stable subspace under $B$), and $S_x=\C^{n_x}$ for $x\neq u$.
So  for any minimal vertex $u$,  $\CK_{min}^{(u)}=(K_{x,x\neq u}=\{1,2,\ldots,n_x\},K_{u}=\{1,2,\ldots, n_{u}-1\})$ is always in $\Int_0(Q,\CE)$.
In particular for $Q_{3,3}$ and dimension vector ${\bf n}=[n_1, n_2 , n_3, n_4, n_5, n_6],$
 $$\CK_{\min}^{(1)}=[\{1,2,\ldots,n_1-1\}, \{1,2,\ldots, n_2\} ,\{1,2,\ldots,n_3\}, \ldots, \{1,2,\ldots,n_6\} ]$$
 and  $$\CK_{\min}^{(2)}=[\{1,2,\ldots,n_1\}, \{1,2,\ldots, n_2-1\} , \{1,2,\ldots,n_3\}, \ldots, \{1,2,\ldots,n_6\}]$$ are $Q$-intersecting sets.
 \end{remark}

\bigskip

We will denote by $E(\CJ)$ the linear form  $\lambda \mapsto \sum\limits_{\crampedclap{x\in Q_0}} \sum_{j \in J_x} \lambda_x(j)$
associated to a family of subsets $J_x\subseteq [1,2,\ldots, n_x]$.
Thus the linear form $E(\CK_{\full})$ is the linear form  $\lambda\mapsto \sum\limits_{\crampedclap{x\in Q_0}} \sum_{j=1}^{n_x} \lambda_x(j)$, that is, is the center equation.

As a consequence of Theorem \ref{thm:moment cone and rep theory summary}, the cone $\Cone_Q(\CE)$ is described by the inequalities $E(\CJ)\leq 0$
when   $\CJ \in \Int_0(Q,\CE)$ and the equality $E(\CK_{\full})=0$.
 Thus for an explicit description of $\Cone_Q(\CE)$, we determine the set $\Int_0(Q,\CE)$, by  ``direct observation" for low dimensions,
 or using the algorithm described in the next section.
 We then use elimination of redundant inequalities to obtain the description of $\Cone_Q(\CE)$ by essential inequalities.

\bigskip

{\bf Example $Q_{3,3}$ with dimension vector ${\bm n}=[p,p,p,q,q,q]$.}

Denote by $\Int_0(p,q)$ the  Horn set  $\Int_0(Q_{3,3},{\bm n})$ (Definition \ref{def:Hornset}).
Let $\CA$ be the subset consisting of the  $4$  arrows $1\to 3$, $2\to 3$  and $4\to 5$, $4\to 6$.
We denote by $\Int_0(\CA, p,q)$ the set $\Int_0(\CA, Q_{3,3},{\bm n})$.
It consists of $Q$-intersecting  subsets $\CK=[K_1,K_2,K_3,K_4,K_5,K_6]$ such that the cardinal of the sets $K_1,K_2,K_3$ coincide, as well as the
cardinal of the sets $K_4,K_5,K_6$. Lemma \ref{equaldim}  reduces the description of $\Int_0(p,q)$  to the description  $\Int_0(\CA, p,q)$, so we call this last set {\bf Non trivial}.

There are two obvious  symmetries by exchanging $1<>2$ and $5<>6,$  as they correspond to exchange vertices $1<>2$ and $5<>6$ in the quiver.
So the sets $\Int_0(p,q)$ are invariant by this group of symmetries.

We also denote by $\Sigma_3\times \Sigma_3$  the group of symmetries
 permuting the vertices $1, 2, 3$  as well as the vertices  $4, 5, 6$.
The set $\Int_0(\CA,p,q)$ is invariant by $\Sigma_3 \times \Sigma_3$, as proved the next section.

\bigskip

{\bf Example  $Q_{3,3}$ with $\bm n=[2,2,2,1,1,1] $}

Here is the list of  sets in $\Int_0(2,1)$.

{\bf Full}
\begin{align*}
[\{1, 2\}, \{1, 2\}, \{1, 2\}, \{1\}, \{1\}, \{1\}] &\\
\end{align*}

{\bf Extremal}
\begin{align*}
\CK_{\max}^{(5)}=&[\{\}, \{\}, \{\}, \{\}, \{1\}, \{\}], &\\
\CK_{\max}^{(6)}=&[\{\}, \{\}, \{\}, \{\}, \{\}, \{1\}], &\\
\CK_{\min}^{(1)}=&[\{1\}, \{1, 2\}, \{1,2\}, \{1\}, \{1\}, \{1\}], &\\
\CK_{\min}^{(2)}=&[\{1,2\}, \{1\}, \{1,2\}, \{1\}, \{1\}, \{1\}] .&\\
\end{align*}

{\bf  Non Trivial}
\begin{align*}
 [\{2\}, \{2\}, \{2\}, \{1\}, \{1\}, \{1\}],&\\
  [\{1, 2\}, \{1, 2\}, \{1, 2\}, \{\}, \{\}, \{\}],&\\
 [\{2\}, \{1\}, \{2\}, \{\}, \{\}, \{\}],&\\
 [\{1\}, \{2\}, \{2\}, \{\}, \{\}, \{\}],&\\
 [\{2\}, \{2\}, \{1\}, \{\}, \{\}, \{\}].&\\
\end{align*}

Remark that the set  {\bf Non Trivial} $=\Int_0(\CA,2,1)$ is invariant by $\Sigma_3\times \Sigma_3$.

\bigskip

Let ${\bf{ a}}=[a_1,a_2,a_3,a_4,a_5,a_6]$  as in Example \ref{u21intro}.
The inequalities $E(\CJ)$ for $\CJ$ belonging to the  set {\bf Extremal} imply the equations
$$  a_5(1)\leq 0,\,\, a_6(1)\leq 0,$$ and the equations $$a_1(2)\geq 0,a_2(2)\geq 0,$$
using the center equation $E(\CK_{\full})=a_1(1)+a_1(2)+a_2(1)+a_2(2)+a_3(1)+a_3(2)+a_4(1)+a_5(1)+a_6(1)=0.$

The  set {\bf Non Trivial} account for the remaining 5 inequalities in Example \ref{u21intro}.

\bigskip

{\bf Example  $Q_{3,3}$ with $\bm n=[2,2,2,2,2,2] $}\label{Q1Q2}

Here is the list of  sets in $\Int_0(2,2)$.

{\bf Full}
\begin{align*}
[\{1, 2\}, \{1, 2\}, \{1, 2\}, \{1,2\}, \{1,2\}, \{1,2\}] &\\
\end{align*}
{\bf Extremal}
\begin{align*}
\CK_{\max}^{(5)}=&[\{\}, \{\}, \{\}, \{\}, \{1\}, \{\}], &\\
\CK_{\max}^{(6)}=&[\{\}, \{\}, \{\}, \{\}, \{\}, \{1\}], &\\
\CK_{\min}^{(1)}=&[\{1\}, \{1, 2\}, \{1,2\}, \{1,2\}, \{1,2\}, \{1,2\}], &\\
\CK_{\min}^{(2)}=&[\{1,2\}, \{1\}, \{1,2\}, \{1,2\}, \{1,2\}, \{1,2\}] .&\\
\end{align*}

{\bf{Non Trivial}}

$$ [\{1\}, \{2\}, \{2\}, \{\}, \{\}, \{\}],  [\{2\}, \{1
\}, \{2\}, \{\}, \{\}, \{\}], [\{2\}, \{2
\}, \{1\}, \{\}, \{\}, \{\}],$$
$$[\{1, 2\}, \{1, 2\}, \{1, 2\}, \{\}, \{\}, \{\}],$$
$$ [\{1, 2\}, \{1, 2\}, \{1, 2\}, \{1\}, \{2\}, \{2\}], [\{1, 2\}, \{1, 2\}, \{1, 2\}, \{2\}, \{1\}, \{2\}], [\{1, 2\}, \{1, 2\}, \{1
, 2\}, \{2\}, \{2\}, \{1\}],$$
$$  [\{1\}, \{2\}, \{2\}, \{2\}, \{2\}, \{2\}], [\{2\}, \{1\}, \{2\}, \{2\}, \{2\}, \{2\}], [\{2\}, \{2\}, \{1\}, \{2\}, \{2\}, \{2\}],$$
$$  [\{2\}, \{2\}, \{2\}, \{1\}, \{2\}, \{2\}], [\{2\}, \{2\}, \{2\}, \{1\}, \{2\}, \{2\}], [\{2\}, \{2\}, \{2\}, \{2\}, \{2\}, \{1\}].$$

Remark that the set {\bf Non Trivial}  $=\Int_0(\CA,2,2)$ is stable by the group of symmetry $\Sigma_3\times \Sigma_3$.
So to describe  it, we may only describe a list of representatives
 of  $\Int_0(\CA,2,2)$  modulo permutations in $\Sigma_3\times \Sigma_3$.
Here is a  set    of such representatives:

\begin{align*}
[\{1\}, \{2\}, \{2\}, \{\}, \{\}, \{\}],&\\
[\{1, 2\}, \{1, 2\}, \{1, 2\}, \{\}, \{\}, \{\}],&\\
[\{1, 2\}, \{1, 2\}, \{1, 2\}, \{1\}, \{2\},\{2\}],&\\
  [\{1\}, \{2\}, \{2\}, \{2\}, \{2\}, \{2\}],&\\
 [\{2\}, \{ 2\}, \{ 2\}, \{1\}, \{2\},\{2\}].&
\end{align*}

Let us see how we can justify the above list with general arguments.
Consider the quiver  $Q_{3,3}$  with dimension vector ${\bf {n}}=[2,2,2,2,2,2].$
 $Q_{3,3}$ contains a subquiver  $P_1$ isomorphic to the  quiver $\CH_2$
 described in (\ref{eq:quiver}) and   a subquiver $P_2$  where we changed orientations of arrows in $\CH_2$.
  Indeed consider the arrow $a : 3 \rightarrow 4$ and the  two subquivers $P_1$  on the left  with points $1,2, 3$ and  $P_2$  on the right with   points $4,5, 6,$  obtained  by considering  suppressing the arrow $a$  in $Q_{3,3},$ see Fig.\ref{fig-vell2}.
\begin{figure}[h]
\hbox to \hsize\bgroup\hss
\beginpicture
\setcoordinatesystem units <1in, 1in>
\setcoordinatesystem units <1in, 0.6in>
\setplotarea x from -2 to 2, y from -.8 to .8
\multiput {\Large$\bullet$} at -1.22 .72  -1.22 -.72  -.5 0  .5 0  1.22 .72  1.22 -.72 /
\setplotsymbol ({\rm .})
\plot -1.22 .72  -.5 0  /
\plot .5 0  1.22 .72 /
\plot -1.22 -.72  -.5 0 /
\plot  .5 0  1.22 -.72 /
\put{\Large $1$} at -1.28 .88
\put{\Large $2$} at -1.28 -.88
\put{\Large $3$} at -.45 .13
\put{\Large $4$} at .45 .13
\put{\Large $5$} at 1.28 .88
\put{\Large $6$} at 1.28 -.88
\arrow <12pt> [.3,.7] from 1.148 .648 to 1.2 .705
\arrow <12pt> [.3,.7] from 1.148 -.648 to 1.2 -.705
\arrow <12pt> [.3,.7] from -.572 -.072 to -.53 -0.03
\arrow <12pt> [.3,.7] from -.572 .072 to -.53 0.03
\endpicture
\hss\egroup \caption{\label{fig-vell2}}
\end{figure}
Thus the  proper $Q_{3,3}$-intersecting sets  not in {\bf Extremal} are necessarily  (see Lemma \ref{equaldim}) of the form
$\CJ=[I_1,I_2,I_3, J_1,J_2,J_3]$
where $(I_1,I_2,I_3)$ parameterize intersecting Schubert cells of $\Gr(r_1,\C^2)$
and $(J_1,J_2,J_3)$  parameterize intersecting Schubert cells of $ \Gr(r_2,\C^2)$ (with $r_1,r_2\in \{1,0\}$) if $|I_1|=[I_2|=|I_3|=r_1$
and  $|J_1|=[J_2|=|J_3|=r_2$ (where the dimension of the intersection may be greater than $0$).
Short arguments on $\edim_{Q,B}$ allows to eliminate other cases that the ones listed above.
Then it is directly easy to see that any element of the list above gives indeed $Q$-intersecting sets.
For example, let us prove that  $[\{2\}, \{2\}, \{2\}, \{1\}, \{2\}, \{2\}]$ is $Q$-intersecting.
Given $r=(r_a)\in \CH_Q(\CE)$ generic (so we may assume  all the $2\times 2$ matrices $r_a$ invertible),  a subrepresentation
$\CS$  of $r$
in position  $[\{2\}, \{2\}, \{2\}, \{1\}, \{2\}, \{2\}]$ exists and is uniquely determined. Indeed, we are forced to take $S_4=\C e_1$.
Since $\{2\}$ is indexing  the full projective space $\Gr(1,\C^2)$,  there are no restrictions  at the other vertices.
So we construct  $\CS=[S_1,S_2,S_3,S_4,S_5,S_6]$ by taking the images or reciproc images of $S_4$ by consecutive  maps $r_a$ connecting a vertex $i$ to $4$.

\section{The algorithm }\label{algo}
A set $\CK\in \Int(Q,\CE)$ produces always a valid inequality for  the cone $\Cone_Q(\CE)$, namely
the inequality  $$E(\CK)= \sum_{x\in Q_0}\sum_{k \in K_x} \lambda_x(k) \leq 0.$$
Thus to determine the cone $\Cone_Q(\CE)$, we have to determine the sets $\CK\subset_{Q,B} [{\bf n}]$.

Let $\CK=(K_x)_{x\in Q_0}$. Let ${\boldsymbol k}=(k_x)_{x\in Q_0}$  be the family of cardinals of the set $K_x$.
It is easy to see that if $K_x(1) < \dots < K_x(k_x)$ are the elements of~$K_x$, then the dimension of the Schubert variety determined by~$\CK$ is
\label{dimomega}\begin{align}
  \dim{\bm \Omega} = \sum_{x\in Q_0} \sum_{j=1}^{k_x} \left( K_x(j) - j \right).
\end{align}

Any family $\CL \subseteq [\boldsymbol k]$  can be composed  with $\CK$
and give rise to a family $\CT = (T_x)_{x\in Q_0}$ of subspaces $$T_x = \oplus_{j=1}^{l_x} e_{K_x(L_x(j))} \subseteq S_x,$$ where $L_x(1) < \dots < L_x(l_x)$ are the elements of $L_x$ and $l_x = \lvert L_x\rvert$.
We have thus
\begin{align}\label{edimQB}
  \edim_{Q,B}(\CT,\CE)
= \sum_{x\in Q_0} \sum_{j=1}^{l_x} \left( K_x(L_x(j)) - j \right)
\;-\!\!\!\! \sum_{a : x\to y \in Q_1} l_x (n_y - l_y)
\end{align}
with  $l_x=|L_x|$ and $l_y=|L_y|$.

Accordingly,   Theorem \ref{thm:main}  translates into the following inductive numerical criterion:

\begin{theorem}(\cite{BVW})\label{criterion}

$\CK \subseteq_{Q,B} [\boldsymbol n]$ if and only if
\begin{equation}\label{eq:refined}
  \sum_{x\in Q_0} \sum_{j=1}^{l_x} \left( K_x(L_x(j)) - j \right) \;\;\geq \!\!\!\!\sum_{a : x\to y \in Q_1} l_x (n_y - l_y)
\end{equation}
for all $\CL \subseteq_{Q,B} [\boldsymbol k ]$.
\end{theorem}

The criterion in (\ref{eq:refined}) is easy to test numerically.
One may further restrict the families $\CL$ that need to be considered  as we will see in a moment.

In  (\cite{BVW}), we proved that it is sufficient to test on the
$\CL \subseteq_{Q,B} [\boldsymbol k]$,  such that $\edim_{Q,B}(\CL,\boldsymbol k) =0.$

Given a subset $\CA$ of the set of arrows, we defined  $\Int(\CA,Q, \CE)$ to be the set of those  $\CK\in \Int(Q, \CE)$ such that
$|K_x|=|K_y|$ for every arrow $x\to y$ in the subset $\CA\subseteq Q_1$.
Then we proved \cite{BVW} that the set  $\Int(\CA,Q, \CE)$ is determined inductively by testing
on the $\CL \subseteq_{Q,B} [\boldsymbol k]$ satisfying the same dimension condition
 (i.e. $|L_x| = |L_y|$ for every arrow $x\to y\in \CA$) and with $\edim_{Q,B}(\CL,\boldsymbol k) =0.$

\bigskip

We now specialize our result to the quiver~$Q_{3,3}$,  dimension vector $[p,p,p,q,q,q]$ and vertices
$[1,2,3,4,5,6]$.

Consider the set $\CA=\{1\to 3,2\to 3, 4\to 5,4\to 6\}$ of arrows in $Q_{3,3}$.
We define $$\Int(p,q)=\Int(Q_{3,3},[p,p,p,q,q,q]),\,\,\Int_0(p,q)=\Int_0(Q_{3,3},[p,p,p,q,q,q]),$$ $$\Int(\CA,p,q)=\Int(\CA,Q_{3,3},[p,p,p,q,q,q]),\,\,\Int_0(\CA,p,q)=\Int_0(\CA,Q_{3,3},[p,p,p,q,q,q]).$$

Let ${\bm \lambda}=[\lambda_{1},\lambda_{2},\lambda_{3},\lambda_{4},\lambda_{5},\lambda_{6}]$ in the Weyl chamber for  $GL(p)^3\times GL(q)^3$.
Consider the following set  of four inequalities
$${\bf Extremal}=\{\lambda_{1}(p)\geq 0,  \lambda_{2}(p)\geq 0, \lambda_{5}(1)\leq 0,  \lambda_{6}(1)\leq 0\}.$$
If ${\bm \lambda}$ is a weight,
the equations in ${\bf Extremal}$ just say that $\lambda_1,\lambda_2$ index polynomial representations of $GL(p)$, while $ \lambda_5, \lambda_6$   index dual of polynomial representations of $GL(q)$.
\begin{lemma}\label{equaldim}
Let $Q=Q_{3,3}$,  ${\bf n}=[p,p,p,q,q,q]$ and $\CE=[\C^p,\C^p,\C^p,\C^q,\C^q,\C^q]$.
Consider the set $\CA=\{1\to 3,2\to 3, 4\to 5,4\to 6\}$.
Then the cone $\Cone_Q(\CE)$ is the cone defined by the center equation $E(\CK_{\rm full})$,
the inequalities $E(\CK)$ with $\CK\in \Int(\CA,p,q)$
together with the four inequalities in ${\bf Extremal}$.
\end{lemma}

\begin{proof}
Let $U_Q(\CE)=U(p)^3\times U(q)^3$.
Let $\mu: \CH_Q(\CE)\to (\sqrt{-1}\Lie(U_Q(\CE)))^*$ be the moment map.
Let ${\bm \t}$  be the Cartan subalgebra of $U_Q(\CE)$. Then ${\bm \t}=\oplus_{j=1}^6 \t_j$, where $\t_j$ is of dimension $p$ for $j\in \{1,2,3\}$
and of dimension $q$ for $j\in \{4,5,6\}$.
Let $h_k$  be the diagonal $p\times p$ (or $q\times q$) matrix with $1$ at the $k$-th place, $0$ at other places.

Let  $\CJ=[J_1,J_2,\ldots, J_6]$ be  a collection of subsets of integers, with  $J_1,J_2,J_3 \subseteq [p]$ and $J_4,J_5,J_6 \subseteq [q]$.
Let $v$ be a vertex of $Q_{3,3}$. We define $X_{\CJ}$ be the element of $\sqrt{-1}{\bm \t}$
with its component $X_v$ in $\sqrt{-1}\t_v$ given by
 $\sum_{j\in J_v} h_j $
 (if $J_v$ is empty, $X_v=0$).

 For example if $p=2,q=2$ and $\CJ= [\{2\}, \{ 2\}, \{ 2\}, \{1\}, \{2\},\{2\}]$, then
 $$X_{\CJ}=
 \left[\left(
   \begin{array}{cc}
     0& 0 \\
     0 & 1 \\
   \end{array}
 \right),
  \left(
   \begin{array}{cc}
     0 & 0 \\
     0 & 1 \\
   \end{array}
 \right),
  \left(\begin{array}{cc}
     0 & 0 \\
     0 & 1 \\
   \end{array}
 \right),
 \left(
   \begin{array}{cc}
     1 & 0 \\
     0 & 0 \\
   \end{array}
 \right),
 \left(
   \begin{array}{cc}
     0 & 0 \\
     0 & 1 \\
   \end{array}
 \right),
 \left(
   \begin{array}{cc}
     0 & 0 \\
     0 & 1 \\
   \end{array}
 \right)\right].$$

Let ${\bm \lambda}\in (\sqrt{-1}\bm \t)^*$.
Then $E(\CJ)({\bm\lambda})=\langle X_{\CJ},{\bm \lambda}\rangle$.

For ${\bm X}=(X_v)$ in $\sqrt{-1}{\bm \t}$, we consider the subspace  $\CH({\bm X})$ of $\CH_Q(\CE)$  given by
$$\CH({\bm X})=\{r\in \CH_Q(\CE); r_a X_v=X_w r_a\}$$ for $a: v\to w$ in $Q_1$, that is  $\CH({\bm X})$  is the subspace of $\CH_Q(\CE)$ annihilated by the infinitesimal action of ${\bm X}$ on $\CH_Q(\CE)$.
Let $\CJ$ be $Q$-intersecting, then $\{E(\CJ)=0\}\cap \Cone_Q(\CE)$ is a face of $\Cone_Q(\CE)$ since $E(\CJ)\leq 0$ is a valid inequation.
Any point ${\bm \lambda}\in \{E(\CJ)=0\}\cap \Cone_Q(\CE)$  is of the form
$-\mu(r)$ with $r\in \CH(X_\CJ))$. This follows from  properties of walls of a Kirwan polyhedron. See for example (\cite{VW}).

Assume that $\CJ$ is not in $\Int(\CA,Q,\CE)$.
For example $|J_1|\neq |J_3|$, for the arrow $a: 1\to 3$.
Let ${\bm \lambda}=-\mu(r)$ with $r\in \CH(X_\CJ)$. Then the component $r_a$ of $r$ is not invertible, since otherwise the equation
$r_a X_{J_1}=X_{J_3} r_a$ would imply that  $X_{J_1}$ and $X_{J_3}$ are conjugated, and so $|J_1|=|J_3|$.
 But then the component $\lambda_1=-\mu(r)[1]$ is the  non negative Hermitian matrix $r_a^* r_a$  with determinant $0$. It follows that  $\lambda_{1}(p)=0$,  (since $\lambda_{1}(p)\geq 0$),  and the face $\{E(\CJ)=0\}\cap \Cone_Q(\CE)$  is contained in the face  $\{\lambda_{1}(p)=0\}\cap \Cone_Q(\CE)$ of $\Cone_Q(\CE)$.
 The proof is similar in all other cases.
 \end{proof}

Denote also by Extremal the subset

 $${\bf Extremal}=\{\CK_{\max}^{(6)},
\CK_{\max}^{(5)},
\CK_{\min}^{(1)},
\CK_{\min}^{(2)}\}$$  of $\Horn_0(p,q)$  giving rise to the $4$ extremal inequalities listed above.

Let  $\Sigma_3\times \Sigma_3$ be  the group of symmetries
 permuting the vertices $1, 2, 3$  as well as the vertices  $4, 5, 6$.

\begin{proposition}\label{symetry}
We have
$$\Horn(p,q)=\CK_{\full}\cup {\bf Extremal}\cup \Int(\CA,p,q).$$

Furthermore,  $\Int(\CA,p,q)$
is invariant by $\Sigma_3\times \Sigma_3$.
\end{proposition}
\begin{proof}
Let $\CK=[K_1,K_2,K_3,K_4,K_5,K_6]$.
Clearly the numerical quantity $\edim_{Q,B}(\CK,\CE)$ is invariant by $\Sigma_3\times \Sigma_3$ when $|K_1|=|K_2|=|K_3|$ and
$|K_4|=|K_5|=|K_6|$.
Let $\CK$ be such that $|K_1|=|K_2|=|K_3|=p'$ and
$|K_4|=|K_5|=|K_6|=q'$.  To decide if $\CK\in \Int(\CA,p,q)$,
 we need  to test  (as proved in \cite{BVW})  that  $\CL$ composed with $\CK$ is in $\Int(\CA,p,q)$ on  proper subsets $\CL\in \Int(\CA,p',q')$.
So by induction, Theorem \ref{criterion}, and invariance of $\edim_{Q,B}$ by symmetry, we obtain our proposition.

\end{proof}

\section{Holomorphic discrete series and tensor product}\label{ds}

Let $G$ be a connected semisimple Lie group, with finite center and maximal compact subgroup $K$. The representation of $G$ in $L^2(G)$
has a discrete series if and only of $\rank(G)=\rank(K)$. If furthermore $G/K$ is  an hermitian symmetric space,
 within the discrete series, we have a series of representations called the holomorphic  discrete series.
 The realization of a representation belonging to the holomorphic  discrete series is very simple both analytically and algebraically. We will though adopt the algebraic point of view.

We thus assume that $G/K$ is provided with a $G$-invariant complex structure.
We start by introducing some general notation. We denote by capital latin letters the groups involved and by the corresponding german letters the corresponding Lie algebras.
We use the superscript $^*$ for dual spaces and the subscript $_\C$ for complexifications; for example, we write $\g^*$ and $\g_\C$.
 Let $\g=\k \oplus \p$  be the Cartan decomposition of $\g$.
 The complex structure of $G/K$  determines   the $K$-invariant decomposition  $\p_\C=\p_+\oplus \p_-$ where $\p_+$, $\p_-$ are abelian subalgebras of $\g_\C$. There exists $z\in \sqrt{-1}\k$ central in $\k_\C$ such that $\p^+=\{X\in \g_\C, [z,X]= X\}$.
 This element $z$ is uniquely determined modulo an element of the center of $\g_\C$.

\bigskip

{\bf Example $G=U(2,3).$}

The maximal  compact subgroup is $K=U(2)\times U(3)$.
The Lie algebra $\g=\u(2,3)$ in block form is given as
$\g=\left \{\left(\begin{array}{cc} a & b  \\b^* & d\\  \end{array}\right)\right\}$  where $ a\in \u(2), d \in  \u(3)$, $b$ is any   $3\times2$ complex matrix and $b^*$ is the matrix conjugate transpose of $b$.
The Cartan decomposition is given by $\k=\u(2) \oplus \u(3)=\left \{\left(\begin{array}{cc} a & 0  \\0 & d\\  \end{array}\right)\right \}$ and $\p =\left\{\left(\begin{array}{cc} 0 & b  \\b^* & 0\\  \end{array}\right)\right\}.$

The complexification of $\g$ is $\gl(5,\C)$.
We have $\p_\C=\p_{+}\oplus \p_-$
where  a matrix in  $\p_+$ have the block form
$\left(\begin{array}{cc}0 & *  \\0 & 0\\  \end{array}\right)$ and  a matrix in $\p_{-}$  have the block form
$\left(\begin{array}{cc}0 & 0  \\ *& 0\\  \end{array}\right).$
Each $\p_{\pm}$ is stable under the adjoint action of $K_\C=GL(2,\C)\times GL(3,\C)$.
More explicitly $\p_+$
identifies to $$\Hom(\C e_3\oplus\C  e_4\oplus \C e_5 , \C e_1\oplus \C e_2)$$  with action of  $K_\C=GL(2,\C)\times GL(3,\C)$ given by $g_1 r g_2^{-1}$
 if $r\in \Hom(\C e_3\oplus\C  e_4\oplus \C e_5 , \C e_1\oplus \C e_2)$. Similarly  for $\p_-$.


\bigskip

 Let $T$ be a Cartan subgroup of $K$, thus $T$ is also a Cartan subgroup of $G$.
  Let $\Lambda \subset (\sqrt{-1}\t)^{*}$ be the lattice of weights of $T.$ The central element $z$ belongs to $\sqrt{-1}\t$.
  We denote by $W_K$ the Weyl group of $K$.
Let $\Delta=\Delta(\g_\C,\t_\C)$ be the roots of $\g_\C$ with respect to $\t_\C.$ We write   $\g_\C=\t_\C\oplus\sum_{\alpha \in \Delta} \g_\alpha$ for the root space decomposition. If $\alpha \in \Delta$, its coroot  $H_\alpha$  is in $\sqrt{-1}\t$ and satisfies $\alpha(H_\alpha)=2.$ The roots are called compact or noncompact depending whether the root space  is in $\k_\C$ or in $\p_\C.$ We denote by $\Delta_c, \Delta_n$ the set of compact and noncompact roots, with $\Delta_n$ decomposed in $\Delta_n^+\cup \Delta_n^-$, according to the decomposition  $\p_\C=\p_+\oplus \p_-$.
Choose a system  $\Delta_c^+$ of positive compact roots, and let  $\Delta^+=\Delta_c^+\cup \Delta_n^+$.
Let $\lambda\in (\sqrt{-1}\t)^*$ such that $\lambda(H_\alpha)\geq 0$ for all $\alpha\in \Delta^+$. Then
 $\langle\lambda, z\rangle \geq 0$  since $\lambda \in \sum_{\alpha \in \Delta^+}\R_{\geq 0}\alpha$, and $\alpha(z)=0$ for $\alpha$ compact,
 while $\alpha(z)=1$ for $\alpha$ non compact positive.

Let $\lambda\in \Lambda$ be such that $\lambda( H_\alpha)\geq  0$ for all $\alpha\in \Delta_c^+$.
We denote by $V_\lambda$ the irreducible representation of $K$ (or $K_\C$) with highest weight $\lambda$.
 Consider $\pi_\lambda$  the generalized Verma module  defined as $$\pi_\lambda=\CU(\g_\C) \otimes _{\CU(\k_\C\oplus \p_-)} V_\lambda$$ where $V_\lambda$ is  the $K$-module of highest weight $\lambda$ and we extend it to be a $\CU(\p_-)$ module by making $\p_-$ acting trivially (here  $\CU $ stands for the universal enveloping algebra). Thus
  the representation $\pi_\lambda$ has highest weight $\lambda,$ for the system $\Delta_c^+\cup -\Delta_n^+$,
 and as a $K$-module,

 \begin{equation}\label{repK}
  \pi_\lambda=\Sym(\p_+)\otimes V_\lambda.
  \end{equation}

Remark that the action of $z$ in $\pi_\lambda$ have  non negative eigenvalues, namely the set of eigenvalues is  $\{\langle\lambda,z\rangle +n,n\geq 0\}$, and each eigenvalue has finite multiplicity.

If $\rho=\frac12 \sum_{\alpha \in \Delta^+ } \alpha$, $\rho_c= \frac12 \sum_{\alpha \in \Delta_c^+ } \alpha$ , $\rho_n =\frac12 \sum_{\alpha \in \Delta_n^+ }$, then  the infinitesimal character of $\pi_\lambda$ is $\lambda+\rho_c-\rho_n.$
If $\langle \lambda+\rho_c-\rho_n, H_\alpha \rangle >0$ for all $\alpha\in \Delta^+$, then
Harish-Chandra has constructed a unitary irreducible representation of $G$ with corresponding $(\g,K)$ module $\pi_\lambda$.
Such representation belongs to the $\bf{ holomorphic \ discrete \ series}$ and we denote it still by $\pi_\lambda.$
The name holomorphic comes from the fact that it can be realized as a subspace of the  space of holomorphic sections of the holomorphic bundle $G\times_K V_\lambda$ on $G/K$,   the condition $\langle\lambda+\rho_c-\rho_n, H_\alpha\rangle >0$ assuring the existence of square integrable holomorphic sections.

{\bf Example $U(p,q)$} \label{exUpq}
Consider   $U(p,q)$, with $\sqrt{-1}\t$ consisting of  diagonal  $(p+q)\times (p+q)$ matrices with real entries.
Let $h_k$  be the diagonal $(p+q)\times (p+q)$ matrix with $1$ at the $k$-th place, $0$ at other places. Then $(h_k)_{k=1}^{p+q}$ is a basis of  $\sqrt{-1}\t$.
We denote by $\epsilon^k$ the  dual basis.
The Weyl group $W_K$ is the product of the permutation groups $\Sigma_p\times \Sigma_q$.

 Consider the standard system of positive roots $$\Delta^{+}=\{\epsilon^i-\epsilon^j, 1 \leq i<j \leq p+q\}.$$
 Then $$\Delta_c^+=\{\epsilon^i-\epsilon^j, 1 \leq i<j \leq p\}\cup \{\epsilon^i-\epsilon^j, p+1\leq i<j \leq p+q\}.$$
  $$\Delta_n^+=\{\epsilon^i-\epsilon^j, 1 \leq i \leq p, p+1\leq j \leq p+q\}.$$

 Then $\rho_n=\frac{1}{2}(q \sum_{i=1}^p \epsilon^i -p\sum_{j=1}^{q-p} \epsilon^{j+p})$
and $(\rho_n-\rho_c)(h_p-h_{p+1})=(p+q)-1$.
Let $\lambda=\sum_{i=1}^{p+q}\lambda_i \epsilon^i$ in $(\sqrt{-1}\t)^*$ with integral coefficients.
The condition $\langle\lambda+\rho_c-\rho_n,H_\alpha\rangle > 0$ for all compact roots gives
separately
$$\lambda_1\geq \lambda_2\geq\cdots\geq  \lambda_p,$$
$$\lambda_{p+1}\geq \lambda_{p+2}\geq \cdots\geq \lambda_{p+q},$$
and  $\langle \lambda+\rho_c-\rho_n,H_{\epsilon^p-\epsilon^{p+1}}\rangle > 0$
gives the condition  $\lambda_p > \lambda_{p+1}+p+q-1$.

\bigskip

Let $\lambda,\mu,\nu$ be dominant weights for $K$.
Define $m_D^\nu (\lambda, \mu) $ to be the multiplicity of the representation $V_\nu$ in the representation   $V_\lambda \otimes  V_{\mu}\otimes \Sym(\p_+)$ of  $K$

Recall the following theorem \cite{Repka},\cite{JakobsenVergne}.

\begin{theorem}\label{decdiscrete} The tensor product of two representations  $\pi_\lambda, \pi_{\mu}$ belonging to the holomorphic discrete series   is a direct sum, with finite multiplicities,  of representations belonging to the holomorphic discrete series:
$$\pi_\lambda \otimes \pi_{\mu}=\otimes_{\nu} m_{\hol}^{\nu}(\lambda,\mu) \pi_\nu.$$
We have:
\begin{equation}\label{tensorproduct} m_{\hol}^\nu(\lambda,\mu )=m_D^{\nu}(\lambda,\mu).\end{equation} \end{theorem}

In other words, if the parameters $\lambda,\mu$  satisfy Harish-Chandra inequalities, and if $m_D^{\nu}(\lambda,\mu)$ is not zero, then
the parameter $\nu$ satisfy also  Harish-Chandra inequalities and   $m_{\hol}^\nu(\lambda,\mu )=m_D^{\nu}(\lambda,\mu).$

(If the parameters $\lambda,\mu$ do not satisfy Harish-Chandra inequalities, then the $(\g,K)$ module  $\pi_\lambda$ may not be irreducible, nor unitarizable. However there is a Jordan-H{\"o}lder series for  $\pi_\lambda\otimes  \pi_{\mu}$ by modules  $\pi_\nu$ and with multiplicity
$m_D^{\nu}(\lambda,\mu).$)

 Let us give a rapid check (not a proof, but the proof is not difficult) of the consistency of the formula by  restricting both members to $K$. The decomposition of the left hand side of Eq.\ref{tensorproduct} restricted to $K$ is

$$(\pi_\lambda \otimes \pi_{\mu})_{|_\k}=\Sym(\p_+)\otimes V _\lambda\otimes \Sym(\p_+)\otimes V_{\mu}.$$

The decomposition of the right hand side is
$$\sum_\nu m_D^{\nu}(\lambda,\mu )\pi_\nu|_K=
\sum_\nu m_D^{\nu}( \lambda,\mu ) \left(V_\nu\otimes \Sym(\p_+)\right)$$
$$=
 \left(\sum_\nu m_D^{\nu}(\lambda,\mu ) V_\nu\right)\otimes \Sym(\p_+).$$

Now  $$\sum_\nu m_D^{\nu}(\lambda,\mu ) V_\nu=V_\lambda\otimes V_{\mu}\otimes \Sym(\p_+)$$
 by definition of $m_D^{\nu} (\lambda,\mu )$. So the right hand side is

$$V_\lambda\otimes V_{\mu}\otimes \Sym(\p_+)\otimes \Sym(\p_+)$$ and thus Eq. \ref{tensorproduct} is true with respect to the $K$ action.
\section{ Geometric analogue}\label{orbit}

As in the preceding section, we consider an hermitian symmetric space $G/K$, and the choice
of the positive root system $\Delta^+=\Delta_c^+\cup \Delta_n^+$.
Consider $\c_{\geq 0}=\{\lambda\in (\sqrt{-1}\t)^*, \lambda(H_\alpha)\geq 0; \alpha\in \Delta^+\}$, the corresponding Weyl chamber.
{\bf In this section, we do not assume $\lambda$ in the lattice of weights}.
If $\lambda\in \c_{\geq 0}$ we denote by $\CO^{nc}_\lambda\subset (\sqrt{-1}\g)^*$ its coadjoint orbit under the action of $G$.
Moment maps for $G$-Hamiltonian spaces (or $K$-Hamiltonian spaces) are valued in $(\sqrt{-1}\g)^*$ (or $(\sqrt{-1}\k)^*$).
In particular, if  the compact group $K$ acts on a complex  vector space $E$  provided with a $K$-invariant Hermitian  form $\langle,\rangle$ that we take antilinear in the first variable,
 the moment map  $\mu: E\to (\sqrt{-1}\k)^*$  is given by $\mu(r)(X)=\langle  r, X r\rangle$, for $r\in E$ and $X\in (\sqrt{-1}\k)$ .

 Recall that $\CO^{nc}_\lambda$  is a (non compact)  $G$-Hamiltonian manifold, the moment map $\CO^{nc}_\lambda\to (\sqrt{-1}\g)^*$ being the inclusion. Since $\CO^{nc}_\lambda$ is provided with a $G$-invariant K{\"a}hler structure, we will call
 $\CO^{nc}_\lambda$ an holomorphic orbit.
The set $G\c_{\geq 0}$ is a convex cone  in $(\sqrt{-1}\g)^*$, as proved by  Vinberg (\cite{MR565090})
 (see \cite{MR1337191}). So the sum $\CO^{nc}_\lambda+\CO^{nc}_\mu$ of two holomorphic orbits is an union of holomorphic orbits.

We  denote by $\c^K_{\geq 0}=\{\lambda\in (\sqrt{-1}\t)^*, \lambda(H_\alpha)\geq 0; \alpha\in \Delta_c^+\}$,
the Weyl chamber for $K$.

Let $z$ in the center of $\sqrt{-1}\k$ such that $\p^+=\{X\in \p_\C, [z,X]=X\}$.
If $f\in  (\sqrt{-1}\g)^*$, the value $\langle f,z\rangle$ is real.
The geometric analogue of the fact that, in a representation $\pi_\lambda$ of the holomorphic discrete series,  the eigenvalues of $z$ are  positive
 and the corresponding eigenspace is with finite multiplicity is the  content of
the following proposition (see \cite{MR3646030}).

\begin{proposition}\label{proper}
Let $\lambda\in \c_{\geq 0}$,  let $\CO^{nc}_\lambda$ its coadjoint orbit, and let $f\in \CO^{nc}_\lambda\subset (\sqrt{-1}\g)^*$.
Then
\begin{enumerate}
\item $ \langle f,z\rangle\geq  \langle \lambda,z\rangle\geq 0$.

\item  For any $t\in \R$, the set $$\CO^{nc}_\lambda(t)=\{\xi\in \CO^{nc}_\lambda,  \langle f,z\rangle=t\}$$
is a compact subset of $\CO^{nc}_\lambda$.

\end{enumerate}

In other words, the function $p:\CO^{nc}_\lambda\to \R$ given by $p(f)=\langle f,z\rangle$ is bounded from below by
$\langle \lambda,z\rangle\geq 0$. Furthermore the fiber of $p$ is compact.

\end{proposition}
The proof  uses the decomposition of $G$ as $KAK$ and  Harish-Chandra description of  a Cartan subalgebra of $\p$ by strongly orthogonal roots.

\bigskip

Let $\lambda,\mu\in \c_{\geq 0}$.
The moment map for the diagonal action of $G$  on the product $\CO^{nc}_\lambda \times \CO^{nc}_\mu$ is the addition  $A:\CO^{nc}_\lambda \times \CO^{nc}_\mu \to \sqrt{-1}\g^*$:
$A(f_1,f_2)=f_1+f_2$.
It is easy to see that $A$ is proper.

The geometric analogue of the fact that a tensor product of representation of the holomorphic discrete series is a direct sum of holomorphic discrete series,
each one occurring with finite multiplicity, is the following proposition, which is a direct consequence of the fact that
$G{\c_\geq0}$ is a convex cone and of Proposition \ref{proper}.

\begin{proposition}
Let $\lambda,\mu$  in $\c_{\geq 0}$, and $\CO^{nc}_\lambda$, $\CO^{nc}_\mu$ be the corresponding coadjoint orbits.
Then
\begin{enumerate}

\item  $\CO^{nc}_\lambda+\CO^{nc}_\mu$ is an union of orbits of elements of $\c_{\geq 0}$.

\item  For any $\xi\in (\sqrt{-1}\g)^*$, the set $$(\CO^{nc}_\lambda\times \CO^{nc}_\mu)(\xi)=\{(f_1,f_2) \in \CO^{nc}_\lambda \times \CO^{nc}_\mu; f_1+f_2=\xi\}$$
is a compact subset of $\CO^{nc}_\lambda\times \CO^{nc}_\mu$.

\end{enumerate}
\end{proposition}

Define $P(\lambda,\mu)=(\CO^{nc}_\lambda+\CO^{nc}_\mu)\cap \c_{\geq 0}$.
It is a set parameterizing the moment map  image $A(\CO^{nc}_\lambda\times \CO^{nc}_\mu)$ modulo the action of $G$.

We now consider the case where $\langle \lambda,H_\alpha\rangle >0, \langle \mu,H_\alpha\rangle >0$ for all $\alpha\in \Delta_n^+$.
Let $O_\lambda, O_\mu $ be the (compact) coadjoint   $K$-orbits of $\lambda, \mu$.
Consider $\p^+$. This  is a $K$-Hamiltonian space with   moment map  $\Phi_+: \p^+\to (\sqrt{-1}\k)^*$.
Thus the product $O_\lambda\times O_\mu \times \p^+$ is  a $K$-Hamiltonian space,  with moment map
$\phi: O_\lambda\times O_\mu \times \p^+\to \sqrt{-1}\k^*$ given by
$\phi(\xi_1,\xi_2,Y)=\xi_1+\xi_2+\Phi_+(Y)$.
Let ${\rm Kir}(O_\lambda\times O_\mu \times \p^+)$ be the intersection of $\phi(O_\lambda\times O_\mu \times \p^+)$
 with the positive Weyl chamber $\c_{\geq 0}^K$.
It is a rational polyhedron.

The following proposition, due to P.E. Paradan \cite{Hornpq}, is the geometric analogue of Theorem \ref{decdiscrete}.
 It  can be deduced from Theorem \ref{decdiscrete} using Paradan's proof of the ``$[Q,R]=0$" theorem in the case of
holomorphic discrete series  \cite{MR3349304}. We sketch another proof based on Fourier transforms of orbits.

\begin{proposition}\label{sumorbits}
Let $\lambda,\mu \in \c_{\geq 0}$ such that  $\langle \lambda,H_\alpha\rangle >0$,  and $\langle \mu,H_\alpha\rangle >0$ for all $\alpha\in \Delta_n^+$.
Then
$$P(\lambda,\mu)={\rm Kir}(O_\lambda\times O_\mu\times \p^+).$$
\end{proposition}

\begin{proof}

We will work with equivariant volumes.
If $M$ is a $G$-Hamiltonian space with Liouville measure $d\beta_M$ and proper moment map $\Phi: M\to (\sqrt{-1}\g)^*$,
the Duistermaat-Heckman measure of the  $G$-Hamiltonian space $M$ is (by definition) $\Phi_*d\beta_M$.
The $G$-equivariant volume of $M$ is
the generalized function of $X\in \g$ defined by
$$I_M(X)=\int_M e^{\langle \Phi(m),X\rangle }d\beta_M= \int_{(\sqrt{-1}\g)^*} e^{\langle \xi,X\rangle }\Phi_*(d\beta_M),$$
so is the Fourier transform of $\Phi_*(d\beta_M)$.

Similarly, if  $N$ is a $K$-Hamiltonian space with Liouville measure $d\beta_N$ and proper moment map $\phi: N\to (\sqrt{-1}\k)^*$,
the Duistermaat-Heckman measure of the  $K$-Hamiltonian space $N$ is  $\phi_*(d\beta_N)$.
The $K$-equivariant volume of $N$ is
the generalized function of $X\in \k$ defined by
$$I_N(X)=\int_N e^{\langle \phi(n),X\rangle }d\beta_N=\int_{(\sqrt{-1}\k)^*} e^{\langle \xi,X\rangle }\phi_*(d\beta_N),$$
so is the Fourier transform of $\phi_*(d\beta_N)$.

In the case of $M=\CO^{nc}_\lambda \times \CO^{nc}_\mu$ (resp. $N=O_\lambda\times O_\mu \times \p^+$),
the measure $\Phi_*(d\beta_M)$  is  easy to compute directly.
Equivalently, it is easy to give explicit formula for its  Fourier transform.
Since the function  $m\mapsto \langle \Phi(m),z\rangle $ on $M$   (resp. $n\mapsto \langle \phi(n),z\rangle $  on $N$)  is a proper map, the  restriction to $\t$  (even to $\R z\subset \t$) of the equivariant volumes of the non compact spaces $M$  or $N$
are well defined as generalized functions.

 We will also use the Fourier transform on $\t$ and the following result.
Let $m$ be a locally polynomial measure on $(\sqrt{-1}\t)^*$ and let us consider
$\CF(m)(X)=\int_{\xi\in (\sqrt{-1}\t)^*} e^{\langle X,\xi\rangle } dm(\xi)$
its Fourier transform as a generalized function on $\t$.
Assume that there exists  a set $\{l_k\in \t^*,k=1,\ldots, s\}$ of non zero linear forms on $\t$
such that $\prod_{k=1}^s l_k(X) \CF(m)(X)$ is analytic.
Assume furthermore that $m$ is supported on the closed halfspace $\langle \xi,z\rangle \geq 0$.
Then the equation $\prod_{k=1}^s l_k\CF(m)=0$ imply that $\CF(m)=0$, thus $m=0$.

\bigskip

Let $\lambda\in \c_{\geq 0}$.
Let $$I_\lambda(X)=\int_{\CO^{nc}_\lambda} e^{ \langle \xi,X\rangle }d\beta_\lambda$$
be the equivariant volume of $\CO^{nc}_\lambda$, that is
the Fourier transform of the orbit $\CO^{nc}_\lambda$ with its Liouville measure $d\beta_\lambda$.
The equivariant volume of the $G$-Hamiltonian space  $\CO^{nc}_\lambda\times \CO^{nc}_\mu$
is the product $I_\lambda(X) I_\mu(X)$ of the equivariant volumes $I_\lambda(X)$, $I_\mu(X)$ of $\CO^{nc}_\lambda$ and $\CO^{nc}_\mu$.
The product is well defined since the addition map $A: \CO^{nc}_\lambda+\CO^{nc}_\mu\to (\sqrt{-1}\g)^*$
is a proper map.
 Let $d\beta=d\beta_\lambda d\beta_\mu$ be the product of the Liouville measures.
 Then
$$I_\lambda(X)I_\mu(X)$$
$$= \int_{\CO^{nc}_\lambda}\int_{\CO^{nc}_\mu}e^{\langle  f_1,X\rangle }e^{\langle f_2,X\rangle} d\beta=
\int_{\CO^{nc}_\lambda\times \CO^{nc}_\mu}e^{\langle  f_1+f_2,X\rangle } d\beta $$
$$=\int_{\CO^{nc}_\lambda\times \CO^{nc}_\mu}e^{\langle  A(f_1,f_2),X\rangle } d\beta =\int_{(\sqrt{-1}\g)^*}e^{\langle \xi,X\rangle } DH(\xi) $$
with  $DH=A_*(d\beta)$ is the Duistermaat-Heckman measure for the
$G$-Hamiltonian space $\CO^{nc}_\lambda\times \CO^{nc}_\mu$ with moment map $A$.
So, by definition of $P(\lambda,\mu)$,  the support  of $DH$  is
$$\CO^{nc}_\lambda+\CO^{nc}_\mu=G\cdot( P(\lambda,\mu))$$
with $P(\lambda,\mu)\subset \c_{\geq 0}\subset (\sqrt{-1}\t)^*$. Remark that for every $\xi\in P(\lambda,\mu)$, $\langle \xi,z\rangle \geq 0$.
We now use the Weyl integration formula on  $G\c_{\geq 0}\subset (\sqrt{-1}\g)^*$.
Desintegrating the measure $DH$ by the Liouville measures of the $G$-orbits $G\xi$ with $\xi\in P(\lambda,\mu)$, we obtain a measure $S$ on
$P(\lambda,\mu)\subset \c_{\geq 0}\subset (\sqrt{-1}\t)^*$.
We consider the action of the compact Weyl group $W_K$ on $(\sqrt{-1}\t)^*$ and the antiinvariant measure $\ant(S)=\sum_{w\in W_K} \epsilon(w) w\cdot S$ on $(\sqrt{-1}\t)^*$.
Since $z$ is a central element in $\sqrt{-1}\k$, $\ant(S)$ is still supported on $\langle z, \xi\rangle\geq 0$.
 We have  the following formula:
For $X\in \t$:
$$\left(\prod_{\alpha\in \Delta_+}\langle \alpha,X\rangle \right)I_\lambda(X)I_\mu(X)=
\int_{\xi\in (\sqrt{-1}\t)^*}e^{\langle \xi,X\rangle } \ant(S)(\xi).$$

We now consider the $K$-Hamiltonian space  $N=O_\lambda\times O_\mu \times \p^+$, and  its moment map $\phi: N\to (\sqrt{-1}\k)^*$.
We see similarly that
$\langle \phi(n),z\rangle\geq \langle \lambda,z\rangle +\langle \mu,z\rangle $ for any $n\in N$,
and that  the moment map $\phi$ is proper.

Consider the equivariant volume of $N$:
$$I_N(X)=\int_N e^{ \langle X,\phi(n)\rangle} d\beta_N= \int_{\sqrt{-1}\k^*} e^{ \langle X,\xi\rangle} \phi_*(d\beta_N).$$
So $N$ gives rise to a $K$-invariant Duistermaat-Heckman measure $DH_N=\phi_*(d\beta_N)$ on $(\sqrt{-1}\k)^*$ supported on
$\phi(N)$.
By definition of  ${\rm Kir}(O_\lambda\times O_\mu \times \p^+)$,
$$\phi(N)= K\cdot({\rm Kir}(O_\lambda\times O_\mu \times \p^+))$$
with ${\rm Kir}(O_\lambda\times O_\mu \times \p^+)\subset \c_{\geq 0}^K\subset \sqrt{-1}\t^*$.

Similarly desintegrating the measure $DH_N=\phi_*(d\beta_N)$ by the Liouville measures of the $K$-orbits $K\xi$ with
$\xi\in {\rm Kir}(O_\lambda\times O_\mu \times \p^+)$, and using Weyl integration formula on $(\sqrt{-1}\k)^*$,
 we obtain a measure $S'$ supported on
${\rm Kir}(O_\lambda\times O_\mu \times \p^+)$.
Let $\ant(S')$ be the anti-invariant measure  defined by   $\ant(S')=\sum_{w\in W_K} \epsilon(w) w\cdot S'$ on $(\sqrt{-1}\t)^*$.
 We have  the following formula.
For $X\in \t$:
$$\prod_{\alpha\in \Delta_c^+}\langle \alpha,X\rangle I_N(X)=\int_{(\sqrt{-1}\t)^*} e^{\langle \xi,X\rangle } \ant(S')(\xi) .$$

Both measures $\ant(S)$ and $\ant(S')$ are supported on the closed halfspace $\langle \xi,z\rangle \geq 0$ and are locally polynomial measures on $(\sqrt{-1}\t)^*$. It is equivalent to prove that $S=S'$ or that $\ant(S)=\ant(S')$ since $S,S'$ are both supported in the Weyl chamber $\c^K_{\geq 0}$, so no cancelation can occur.
It is easy to see that  both  $(\prod_{\alpha\in \Delta^+}\langle \alpha,X\rangle) \CF(S)(X)$
 and $(\prod_{\alpha\in \Delta^+}\langle \alpha,X\rangle) \CF(S')(X)$ are analytic (see for example Proposition 32 of \cite{MR1095340}).
We thus can prove that $S=S'$ by proving that $\CF(\ant(S))(X)=\CF(\ant(S'))(X)$
on the open set of   $X\in \t$ such that $\prod_{\alpha\in \Delta^+}\langle \alpha,X\rangle \neq 0$.
Now, for such $X$, all equivariant volumes occurring can be computed by an (easy) case of the Berline-Vergne localization formula
 for the corresponding possibly non compact
 Hamiltonian spaces with proper moment maps.
We give the formulae in the case where $\langle \lambda,H_\alpha\rangle\neq 0$  and $\langle \mu,H_\alpha\rangle\neq 0 $ for all $\alpha\in \Delta^+$, since they are easier to state.

%
We have  by Rossmann formula  (\cite{BV82}):

$$I_\lambda(X) I\mu(X)$$
$$=\left(\sum_{w_1\in W_K} \epsilon(w_1)\frac{e^{\langle w_1\lambda,X\rangle }}{\prod_{\alpha\in \Delta^+}\langle \alpha, X\rangle }\right)
\left(\sum_{w_2\in W_K}\epsilon(w_2) \frac{e^{\langle w_2\mu,X\rangle }}{\prod_{\alpha\in \Delta^+}\langle\alpha, X\rangle }\right).$$

Thus we obtain:
\begin{equation}\label{twoorbits}
\CF(\ant(S))(X)= \frac{\sum_{w_1,w_2\in W_K\times W_K}\epsilon(w_1)\epsilon(w_2)e^{\langle w_1\lambda+w_2\mu, X\rangle }}{\prod_{\alpha\in \Delta^+}\langle \alpha, X\rangle } .
\end{equation}
We now compute the equivariant volume of $N$, that is the product of the equivariant volumes of the compact $K$-orbits $O_\lambda$, $O_\mu$  and of
 the symplectic space $\p^+$. The list of weights of the action of  $T$ on $\p^+$ is  the list  $\Delta_n^+$.
 The equivariant volume of $\p^+$ is thus given, for $X\in \t$ such that $\prod_{\alpha\in \Delta_n^+} \langle \alpha,X\rangle \neq 0$, by
 $\frac{1}{\prod_{\alpha\in \Delta_n^+} \langle \alpha,X\rangle}$. The  equivariant volumes of the compact $K$-orbits $O_\lambda$, $O_\mu$
 are given by Harish-Chandra formula.
 So we obtain
$$\CF(\ant(S'))(X)= \frac{1}{\prod_{\alpha\in \Delta_n^+} \langle \alpha,X\rangle} \left(
 \frac{\sum_{w_1,w_2\in W_K\times W_K}\epsilon(w_1)\epsilon(w_2)e^{\langle w_1\lambda+w_2\mu, X\rangle }}{\prod_{\alpha\in \Delta_c^+}\langle \alpha, X\rangle }\right)$$
and we conclude that $\CF(\ant(S))(X)=\CF(\ant(S'))(X)$  on the open set of   $X\in \t$ such that $\prod_{\alpha\in \Delta}\langle \alpha,X\rangle \neq 0$,
and so everywhere.
The measure $S$ and $S'$ are thus equal, so their support $P(\lambda,\mu)$ and ${\rm Kir}(O_\lambda\times O_\mu \times \p^+)$  are equal.
\end{proof}

Denote by $\Cone_{\hol}(G)\subset \c_{\geq 0}\oplus \c_{\geq 0}\oplus \c_{\geq 0}$
to be the set of triples  $(A,B,C)$ such that $\CO_{C}^{nc}\subset \CO_{A}^{nc}+\CO_B^{nc}$.

An important corollary of the discussion above is the following proposition
\begin{proposition}(\cite{Hornpq})\label{compare}
The cone $\Cone_{\hol}(G)$ is the cone generated by
the weights $(\lambda,\mu,\nu)$
such that $m_D^{\nu}(\lambda,\mu)>0$.
\end{proposition}
\begin{proof}
It follows from  theorems on Kirwan polyhedron (\cite{NessMumford84},\cite{GS1982qr}) for representations of compact groups that
$\{(\lambda,\mu,\nu); \nu\in  {\rm Kir}(O_\lambda\times O_\mu\times \p^+)\}$
is the polyhedron generated by the triples  of dominant weights  $(\lambda,\mu,\nu)$
such that  $V_\nu\subset V_\lambda\otimes V_\mu\otimes \Sym(\p^+)$.
\end{proof}

So as it should be, by Proposition \ref{decdiscrete},
the cone  $\Cone_{\hol}(G)$ in $\c_{\geq 0}\oplus \c_{\geq 0}\oplus \c_{\geq 0}$
 generated by  the $(\lambda,\mu,\nu)$ such that $\CO_\nu\subset \CO_\lambda+\CO_\mu$ coincide with the cone generated by
 $(\lambda,\mu,\nu)$  parameters for the holomorphic discrete series with $m_{\hol}^{\nu}(\lambda,\mu)>0$.
 This is in accordance with the general philosophy $[Q,R]=0$ (and conversely is a corollary of $[Q,R]=0$ for holomorphic quantizable  orbits, as proved by P-E-Paradan (\cite{MR3349304})).

\begin{remark}
{\em Consider the $G$-Hamiltonian space $\CO^{nc}_\lambda \times \CO^{nc}_\mu$
and the induced $G$-Hamiltonian space $G\times_K (O_\lambda\times O_\mu\times \p^+)$.
A more satisfying proof of Proposition \ref{sumorbits} would be to prove (in the spirit of Deltour \cite{Deltour}) that these two $G$-Hamiltonian spaces are isomorphic.}
\end{remark}
\section{The quiver representation}\label{quiverrep}

The aim of this section is to relate   tensor products of  holomorphic discrete series for $U(p,q)$
with representations of the quiver $Q_{3,3}$.
We consider the quiver $Q_{3,3}$  described in Fig.\ref{fig-vell}, but now label the vertices as $x_1,x_2,x_3,y_1,y_2,y_3$. A couples $(x_i,y_i)$
will parameterize an object for  $i$-th copy of $U(p,q)$ in $U(p,q)^3$. See \ref {fig-vell}.
\begin{figure}[h]
\hbox to \hsize\bgroup\hss
\beginpicture
\setcoordinatesystem units <1.2in, 0.9in>

\setplotarea x from -2 to 2, y from -.8 to .8
\multiput {\Large$\bullet$} at -1.22 .72  -1.22 -.72  -.5 0  .5 0  1.22 .72  1.22 -.72 /
\setplotsymbol ({\rm .})
\plot -1.22 .72  -.5 0  .5 0  1.22 .72 /
\plot -1.22 -.72  -.5 0 /
\plot  .5 0  1.22 -.72 /
\put{\Large $x_1$} at -1.28 .88
\put{\Large $x_2$} at -1.28 -.88
\put{\Large $x_3$} at -.45 .13
\put{\Large $y_3$} at .45 .13
\put{\Large $y_1$} at 1.28 .88
\put{\Large $y_2$} at 1.28 -.88
\arrow <12pt> [.3,.7] from -.4 0 to -.47 0
\arrow <12pt> [.3,.7] from 1.148 .648 to 1.2 .705
\arrow <12pt> [.3,.7] from 1.148 -.648 to 1.2 -.705
\arrow <12pt> [.3,.7] from -.572 -.072 to -.53 -0.03
\arrow <12pt> [.3,.7] from -.572 .072 to -.53 0.03
\endpicture
\hss\egroup \caption{\label{fig-vell}}
\end{figure}

Given an integer $s$, we label the dominant weights for~$U(s)$ by a
sequence $\lambda$ of $s$ slowly decreasing integers $(\lambda_1\geq\cdots \geq \lambda_s)$.
If $\lambda_s\geq 0$, we simply write $\lambda \geq 0$.
If $\lambda_1\leq 0$, we write $\lambda\leq 0$.
Let $V_\lambda$ be the representation of $U(s)$ with highest weight $\lambda$.
The dual representation $V_\lambda^*$ is indexed by
$\lambda^*=(-\lambda_s\geq\cdots \geq -\lambda_1)$.
Remark that if $\lambda\geq 0$, then $\lambda^*\leq 0.$

We choose $E_{x_i}=\C^p$, $E_{y_i}=\C^q$, and  $\CE=(E_{x_1},E_{x_2},E_{x_3}, E_{y_1},E_{y_2},E_{y_3} )$.
The space  $H(Q_{3,3})$ associated to $\CE$
is $$\oplus_{i=1}^2 \Hom(E_{x_i},E_{x_3})\oplus \Hom(E_{y_3},E_{x_3}) \oplus \oplus_{j=1}^2 \Hom(E_{y_3},E_{y_i}).$$
Then  $\Sym^*(H(Q_{3,3}))$ is a representation space for
 $GL(\CE)=\GL(E_{x_1}) \times  GL(E_{x_2})\times GL(E_{x_3}) \times GL(E_{y_1}) \times GL(E_{y_2})\times GL(E_{y_3})$.

Write its decomposition
$$\Sym ^*(H(Q_{3,3}))$$$$=\oplus _{\alpha_1,\alpha_2,\alpha_3,\beta_1,\beta_2,\beta_3}m_Q(\alpha_1,\alpha_2,\alpha_3,\beta_1,\beta_2,\beta_3) V_{\alpha_1}^{x_1}\otimes
V_{\alpha_2}^{x_2}\otimes V_{\alpha_3}^{x_3}\otimes V_{\beta_1}^{y_1}\otimes V_{\beta_2}^{y_2}\otimes V_{\beta_3}^{y_3}$$
where $\alpha_1,\alpha_2,\alpha_3$ are dominant weights for $U(p)$,
while $\beta_1,\beta_2,\beta_3$ are dominant weights for $U(q)$.

For $\lambda,\mu,\nu$,  a triple of dominant weights of $U(p)\times U(q)$  (so $\lambda$  is a couple $(\alpha,\beta)$  where $\alpha$ is a dominant weight for  $U(p)$ and $\beta$ for $U(q)$, etc..), recall that
$m_D^\nu(\lambda,\mu)$
is the multiplicity of  $V_\nu$ in $V_\lambda\otimes V_{\mu}\otimes \Sym(\p^+)$, where
 $\p^+=\Hom(\C^q,\C^p)$.

\begin{lemma}\label{mQandmD}
If $m_Q(\alpha_1,\alpha_2,\alpha_3,\beta_1,\beta_2,\beta_3)>0$, then $\alpha_1\geq 0$, $\beta_1\leq 0$, $\alpha_2\geq 0$, $\beta_2\leq 0$, and $\alpha_3\leq 0, \beta_3\geq 0$.
Furthermore,
 $$m_Q(\alpha_1,\alpha_2,\alpha_3,\beta_1,\beta_2,\beta_3)= m_D^\nu(\lambda,\mu),$$
where
 $$V_\lambda=V_{\alpha_1}\otimes V_{\beta_1},\hspace{1cm} V_\mu=V_{\alpha_2}\otimes V_{\beta_2},\hspace{1cm}V_\nu=V_{\alpha_3^*}\otimes V_{\beta_3^*}.$$
\end{lemma}

    \begin{proof}
 \end{proof}
We give the proof in the case where $p\leq q$.
We write
$$\Sym^*(\CH_{Q_{3,3}}) =A\otimes B\otimes C$$  with
$$A=\Sym^*(\Hom(E_{x_1},E_{x_3})) \otimes    \Sym^*(\Hom(E_{x_2},E_{x_3})),$$
$$B= \Sym^*(\Hom(E_{y_3},E_{x_3})),$$
$$C= \Sym^*(\Hom(E_{y_3},E_{y_1}))\otimes \Sym^*( \Hom(E_{y_3},E_{y_2})).$$

We recall the Cauchy formula : \begin{lemma}
  Let $N, n$ be positive integers, and assume that $N\geq n$.
  The decomposition of $\Sym(\C^{n}\otimes \C^N)$ with respect to the natural action of $U(n)\times U(N)$ is given by the \emph{Cauchy formula}
  \begin{equation}\label{exa:Cauchy formula}
    \Sym(\C^{n}\otimes \C^N)=\bigoplus_{\nu\geq 0 }V_{\nu}^{U(n)}\otimes V_{\tilde\nu}^{U(N)}.
  \end{equation}
Here $ \tilde\nu$  is the sequence $\nu $  to which we add on the right $N-n$ zeros.

\end{lemma}
 Using the identification $$\Hom(V,W) =V^*\otimes W$$  and looking at the decomposition of $A$, $C$, above,  we obtain
 $$\Sym^*(\CH_{Q_{3,3}})$$

$$\sum_{\alpha_1\geq 0,\alpha_2\geq 0, \beta_1\leq 0, \beta_2\leq 0}
V^{x_1}_{\alpha_1}\otimes V^{x_2}_{\alpha_2}\otimes {\rm Middle}_{\alpha_1,\alpha_2,\beta_1,\beta_2} \otimes V^{y_1}_{\beta_1}\otimes V^{y_2}_{\beta_2}$$
where   ${\rm Middle}_{\alpha_1,\alpha_2,\beta_1,\beta_2}$ is the representation of $U(E_{x_3})\times U(E_{y_3})$
equal to $$V^{x_3}_{\alpha_1^*}\otimes V^{x_3}_{\alpha_2^*}\otimes \Sym^*(\Hom(E_{y_3},E_{x_3}))\otimes  V^{y_3}_{\beta_1^*}\otimes V^{y_3}_{\beta_2^*}.$$
 So
$${\rm Middle}_{\alpha_1,\alpha_2,\beta_1,\beta_2}=V_{\lambda}^*\otimes V_\mu^* \otimes \Sym^*(\Hom(E_{y_3},E_{x_3}))$$
$$=(V_{\lambda}\otimes V_\mu \otimes \Sym(\Hom(E_{y_3},E_{x_3}))^*,$$
with $$V_\lambda=V_{\alpha_1}\otimes V_{\beta_1},\hspace{1cm} V_\mu=V_{\alpha_2}\otimes V_{\beta_2}.$$
By definition of $m_D$, this is
$$\sum_{\nu} m_D^{\nu}(\lambda,\mu) V_\nu^*=\sum_{\nu} m_D(\nu^*,\lambda,\mu) V_\nu$$

Replacing ${\rm Middle}_{\alpha_1,\alpha_2,\beta_1,\beta_2}$ by this last expression, we obtain our formula.

\section{ Cone inequalities  and comparison with P.E. Paradan's  list}\label{coneeq}

Let $G=U(p,q)$ and $n=p+q$.
Let $K=U(p)\times U(q)$ be the maximal compact subgroup of $G$.

Let $\h$ be the Cartan subalgebra of $\Lie(G)$ consisting of diagonal matrices with imaginary entries.
An element  $\xi\in (\sqrt{-1}\h)^*$ is written as $\sum_{i=1}^{p+q} \xi(i)\epsilon^i$  with $\xi(i)$  reals, using the notations of Example   \ref{exUpq}.
Let $$\c_{\geq 0}=\{\xi;   \xi(1)\geq \cdots \geq \xi(p)\geq \xi(p+1)\geq \cdots \geq \xi(p+q)\}.$$
 Recall that the cone $\Cone_{\hol}(p,q)$ consists of the triples  $(A,B,C)$ such that $\CO_C^{nc}\subset \CO_A^{nc}+\CO_B^{nc}$.
 Equivalently (following \cite{Hornpq} and Proposition \ref{compare}), it is the cone generated by the  weights $\lambda,\mu,\nu$ in
$\c_{\geq 0}\oplus \c_{\geq 0}\oplus \c_{\geq 0}$ such that
$m_D^{\nu}(\lambda,\mu)>0$.

Let  $\CJ=[J_1,J_2,\ldots, J_6]$ be  a collection of subsets of integers, with  $J_1,J_2,J_3 \subseteq [p]$ and $J_4,J_5,J_6 \subseteq [q]$.
Consider a triple $(A,B,C)$ in $(\sqrt{-1}\h)^*\oplus (\sqrt{-1}\h)^*\oplus (\sqrt{-1}\h)^*$.
We define  the linear form $e(\CJ)$ by
$$e(\CJ)(A,B,C)$$
$$=\sum_{i\in J_1} A(i)+\sum_{i\in J_2} B(i)-\sum_{i\in J_3} C(p+1-i)
-\sum_{i\in J_4} C(p+q-i+1)+
\sum_{i\in J_5} A(p+i)+\sum_{i\in J_6} B(p+i).$$

We now consider the quiver $Q_{3,3}$   and index its vertices by $\{1,2,3,4,5,6\}$.
A family of objects indexed  by $Q_0$ is written as a list in the order $1\to 6$.
So an element $[J_1,J_2,J_3,J_4,J_5,J_6]\in \Int_0(\CA,p,q)$ is a sequence of sets such that $J_1,J_2,J_3 \subseteq [p]$ and $J_4,J_5,J_6 \subseteq [q]$ with
$|J_1|=|J_2|=|J_3|$ and  $|J_4|=|J_5|=|J_6|$.

\begin{proposition}
The cone $C_{\hol}(p,q)\subset  \c_{\geq 0}\oplus \c_{\geq 0}\oplus \c_{\geq 0}$
is the cone defined by the equation
$$\sum_{i=1}^{p+q} A(i)+\sum_{i=1}^{p+q} B(i)=\sum_{i=1}^{p+q} C(i)$$
and the inequalities $e(\CJ)(A,B,C)\leq 0$
for all $\CJ\in \Int_0(\CA,p,q)$.
\end{proposition}
\begin{proof}\end{proof}
Let $(A,B,C)$ in $\c_{\geq 0}\oplus \c_{\geq 0}\oplus \c_{\geq 0}$.
Let $t\in \R$ be a real number.
Consider  $t\Id_n=t\,
(\sum_{k=1}^n h_k)$ in $\sqrt{-1}\Lie(G)$. It is in the center of  $\Lie(G)_\C$.
Thus $r_t=\sum_{k=1}^n t\epsilon^k$ is invariant by the coadjoint action of $U(p,q)$.
Thus we see that the translations  $$(A,B,C)\to (A+r_{t_1},B+r_{t_2},C-r_{t_1}-r_{t_2})$$
 leaves the cone $C_{\hol}(p,q)$ invariant, for any $(t_1,t_2)\in \R^2$.

Since the sets $J_1,J_2,J_3$  as well as the sets $J_4,J_5,J_6$
have same cardinality, the linear form $e(\CJ)$ is invariant by the same two parameter group of translations.
We can find $t_1,t_2$ such that $A=(A(1)\geq \cdots \geq A(p)\geq 0\geq A(p+1)\geq \cdots \geq A(p+q)$
and
$B=(B(1)\geq \cdots\geq B(p)\geq 0\geq B(p+1)\geq \cdots \geq B(p+q)$.
Define \begin{equation} \label{transform}
\begin{array} {l} a_1=(A(1),\ldots, A(p) ),\\
 a_5= (A(p+1),\ldots,  A(p+q)),\\
a_2=(B(1),\ldots,   B(p) ),\\
a_6= (B(p+1), \ldots,  B(p+q)), \\
a_3=(-C(p),\ldots, -C(1)), \\
 a_4=(-C(p+q),\ldots,-C(p+1)).\\
\end{array}
\end{equation}

 Then ${\bf a}=[a_1,a_2,a_3,a_4,a_5,a_6]$ is in the Weyl chamber for the group $U_Q(\CE)=U(p)^3\times
  U(q)^3$.

Furthermore, by our choice,  $a_1,a_2\geq 0$ and $a_5,a_6\leq 0$. It is easy to see that $e(\CJ)(A,B,C)=E(\CJ)({\bf a})$.
Thus if $(A,B,C)$ satisfy the equations $e(\CJ)\leq 0$, the element ${\bf a}$ is in the cone $\Cone_Q(\CE)$ since the four inequalities in {\bf Extremal} are satisfied by our choice and the other inequalities  describing $\Cone_Q(\CE)$ are $E(\CJ)\leq 0$ for  $\CJ\in \Int_0(\CA,p,q)$
    (Lemma  \ref{equaldim}).
Assume that $(A,B,C)$ are with integral coefficients, then $m_Q(\bf a)>0$,
and so $m_D^{\nu}(\lambda,\mu)>0$ with $V_\lambda=V_{a_1}\otimes V_{a_5}$, $V_\mu=V_{a_2}\otimes V_{a_6}$
$V_\nu=V_{a_3}^*\otimes V_{a_4}^*$,
considered as irreducible representations of $K=U(p)\times
 U(q)$.
From Proposition \ref{parqui}, this in turn implies that $(A,B,C)\in \Cone_{\hol}(p,q)$.
This proves the proposition.

\bigskip

We now give the example of $\Cone_Q(\CE)$ for $Q=Q_{3,3}$ and dimension vector $[2,2,2,2,2,2]$.
The  inequalities corresponding to the elements in $\Int_0(\CA,2,2)$  are thus
\begin{equation}\label{eqconeeq}
\begin{array} {l} a_1(1)+a_2(2)+a_3(2) \leq 0, \\
 a_1(2)+a_2(1)+a_3(2) \leq 0, \\
 a_1(2)+a_2(2)+a_3(1)\leq 0, \\
a_1(1)+a_1(2)+a_2(1)+a_2(2)+a_3(1)+a_3(2) \leq 0, \\
a_1(2)+a_2(2)+a_3(2)+a_4(1)+a_5(2)+a_6(2)\leq 0, \\
a_1(2)+a_2(2)+a_3(2)+a_4(2)+a_5(1)+a_6(2)\leq 0,\\
a_1(2)+a_2(2)+a_3(2)+a_4(2)+a_5(2)+a_6(1) \leq 0,\\
a_1(1)+a_2(2)+a_3(2)+a_4(2)+a_5(2)+a_6(2) \leq 0, \\
a_1(2)+a_2(1)+a_3(2)+a_4(2)+a_5(2)+a_6(2) \leq 0, \\
a_1(2)+a_2(2)+a_3(1)+a_4(2)+a_5(2)+a_6(2) \leq 0, \\
 a_1(1)+a_1(2)+a_2(1)+a_2(2)+a_3(1)+a_3(2)+a_4(1)+a_5(2)+a_6(2) \leq 0, \\
a_1(1)+a_1(2)+a_2(1)+a_2(2)+a_3(1)+a_3(2)+a_4(2)+a_5(1)+a_6(2) \leq 0, \\
a_1(1)+a_1(2)+a_2(1)+a_2(2)+a_3(1)+a_3(2)+a_4(2)+a_5(2)+a_6(1)\leq 0.\\
\end{array}
\end{equation}
 It turns out that all the inequalities are essential.

Using the transformation  defined in (\ref{transform})  $(A,B,C)\to (a_1,a_2,a_3,a_4,a_5,a_6)$, we reobtain that the inequalities of $\Cone_{\hol}(2,2)$ are the ones obtained in \cite{Hornpq} and that we list here:
\begin{eqnarray*}{}
 A_{1}+A_{2}+A_{3}+A_{4}+B_{1}+B_{2}+B_{3}+B_{4} = C_{1}+C_{2}+C_{3}+C_{4},  \nonumber  \\
 A_{1}+A_{2}+B_{1}+B_{2}\le C_{1}+C_{2},     \\
A_{2}+B_{2}\le C_{2}, \ A_{2}+B_{1}\le C_{1} , \ A_{1}+B_{2}\le C_{1},\\
A_{3}+B_{3}\ge C_{3} , \ A_{3}+B_{4}\ge C_{4}, \ A_{4}+B_{3}\ge C_3, \\
A_{2}+A_{4}+B_{2}+B_{4}\le C_{1}+C_{4}, \ \
A_{2}+A_{4}+B_{2}+B_{4}\le C_{2}+C_{3}, \\
 A_{2}+A_{4}+B_{1}+B_{4}\le C_{1}+C_{3}, \ \
A_{1}+A_{4}+B_{2}+B_{4}\le C_{1}+C_{3},\\
A_{2}+A_{4}+B_{2}+B_{3}\le C_{1}+C_{3},\ \
A_{2}+A_{3}+B_{2}+B_{4}\le C_{1}+C_{3}.\\
\end{eqnarray*}

\section {Some more examples}\label{examples}

We give some more examples of $\Horn_0(\CA,p,q)$ which in turn determines the inequalities of the cone $\Cone_Q(\CE)$
or equivalently the inequalities  of the cone $\Cone_{\hol}(p,q)$.
To obtain the full set $\Horn_0(p,q)$, we need to add $\CK_{\full}$  (giving an equality)  and the $4$ additional sets
$\CK_{\max}^{(6)}$,
$\CK_{\max}^{(5)}$,
$\CK_{\min}^{(1)}$,
$\CK_{\min}^{(2)}$ corresponding to the extremal vertices $1,2,5,6$.

Here are  some facts allowing to compute by ``hand"  the examples below.

Recall (Proposition \ref{symetry}) that the set  $\Horn_0(\CA,p,q)$ is invariant by $\Sigma_3\times \Sigma_3$.

Furthermore, recall Remark \ref{subquiver}.
 Let $\CB$ be the set of arrows $1\to 3,2\to 3$ of the quiver $\CH_2$.
 Consider $P_1$ the subquiver of $Q_{3,3}$ with vertices $\{1,2,3\}$, and $P_2$  with vertices $\{4,5,6\}$   described in Fig.\ref{fig-vell2}}.
 So $P_1$ is the quiver  $\CH_2$ and $P_2$ is isomorphic to  $\CH_2$ with opposite orientations.
  If $\CK=[K_1,K_2,K_3,K_4,K_5,K_6]$ are in $\Horn_0(\CA, p,q)$,
and $|K_1|=|K_2|=|K_3|=r_1$,  $|K_4|=|K_5|=|K_6|=r_2$, then  necessarily $\CK_{\lleft}=[K_1,K_2,K_3]$
 is $P_1$-intersecting in $\CE_{P_1}$, and  $\CK_{\rright}=[K_4,K_5,K_6]$ is $P_2$-intersecting in $\CE_{P_2}$.
This means that  the  Schubert varieties determined by $K_1,K_2,K_3$  are intersecting in $\Gr(r_1,p)$, and the Schubert varieties
determined by $K_4,K_5,K_6$ are intersecting in  $\Gr(r_2,q)$.

To help the reader checking our lists, we  list the elements  $\CJ=[J_1,J_2,J_3]\in \Int(\CB,\CH_2, [2,2,2])$
and $\Int(\CB,\CH_2, [3,3,3])$,
 for the quiver $\CH_2$  with $|J_1|=|J_2|=|J_3|=r$
 together with their $\dim=\edim_{Q,B}$ which computes the dimension of the homological intersection of the corresponding Schubert varieties $\Omega_1,\Omega_2,\Omega_3$.

The set  $\Int(\CB,\CH_2,[2,2,2])$ is invariant by permutations of the vertices $1,2,3$ .
So we list representatives grouping  them by the cardinality  $r$ of the sets $J_i$.

$r=1$
$$[\{1\}, \{2\}, \{2\}], \dim=0 $$

$r=2$

$$[\{1,2\}, \{1,2\}, \{1,2\}], \dim=1 $$

The set  $\Int(\CB,\CH_2,[3,3,3])$ is invariant by permutations of the vertices $1,2,3$.
So we list representatives grouping  them by the cardinality  $r$ of the sets $J_i$.
Since the Grassmannians $\Gr(r,n)$ is isomorphic to $\Gr(n-r,r)$,
the list of $\CH_2$-intersecting sets for $r=2$ is immediately deduced from the  list for $r=1$, which is evident
since we are in the projective space.

$r=1$ representatives:

$$[\{1\}, \{3\}, \{3\}], \dim=0,$$
$$[\{2\}, \{2\}, \{3\}], \dim=0,$$
$$ [\{2\}, \{3\}, \{3\}], \dim=1,$$
$$[\{3\}, \{3\}, \{3\}], \dim=2.$$

$r=2$ representatives:

$$[\{1, 2\}, \{2, 3\}, \{2, 3\}], \dim=0, $$
$$[\{1, 3\},  \{1, 3\},  \{2, 3\}], \dim=0,$$
$$ [\{1, 3\}, \{2, 3\}, \{2, 3\}], \dim=1, $$
$$ [\{2, 3\}, \{2, 3\}, \{2, 3\}], \dim=2.$$

We also relate the numerical quantity $\edim_{Q,B}$ for $Q=Q_{3,3}$ to the  $\edim_{Q,B}$ of the subquivers  $P_i$ by the following lemma:
\begin{lemma}
Let $\CE=[\C^p,\C^p,\C^p,\C^q,\C^q,\C^q]$ and  let $\CK\subseteq \CE$.
Then
\begin{align}\label{Q1Q2dim}
\begin{split}\edim_{Q_{3,3},B}(\CK, \CE)=\qquad \qquad \qquad\qquad\qquad\\
 \edim_{P_1,B}(\CK_{\lleft},[p,p,p]) +\edim_{P_2,B}(\CK_{\rright}, [q,q,q] )-|K_{4}| (p -|K_3|).
\end{split}\end{align}

\end{lemma}
\begin{proof}
If $\CK=(K_x)_{x\in Q_0}\subset_{Q,B} [{\bf n}], $
 $k_x=|K_x|$,  and ${\bm \Omega}$ is the Schubert variety determined by $\CK$, the general formula (  see \ref{edimQB})
\begin{align*}
  \edim_{Q_{3,3},B}(\CK,\CE) =\dim {\bm \Omega} - \!\!\!\!\sum_{a\colon x\to y\in Q_1}\!\!\!\! k_x (n_y-k_y).
\end{align*}
implies the formula of the lemma.
\end{proof}

Thus using these properties, it is easy to compute  the following examples.

\bigskip

{\bf {Example $Q_{3,3}$}} with dimension vector $[3,3,3,2,2,2]$

There are $38$ elements in  $\Int_0(\CA,3,2)$.
This set is invariant by the group $\Sigma_3\times \Sigma_3$. So we list only a set of representatives
  up to permutations belonging to $\Sigma_3\times \Sigma_3$.
\begin{flalign*}
[\{1\}, \{3\}, \{3\}, \{\}, \{\}, \{\}],\\
 [\{2\}, \{2\}, \{3\}, \{\}, \{\}, \{\}],\\
[\{2\}, \{3\}, \{3\}, \{2\}, \{2\}, \{2\}],\\
 [\{3\}, \{3\}, \{3\}, \{1\}, \{2\}, \{2\}],\\
[\{1, 2\}, \{2, 3\}, \{2, 3\}, \{\}, \{\}, \{\}], \\
[\{1, 2\}, \{2, 3\}, \{2, 3\}, \{2\}, \{2\}, \{2\}],\\
 [\{1, 3\}, \{1, 3\}, \{2, 3\}, \{\}, \{\}, \{\}],\\
 [\{1, 3\}, \{1, 3\}, \{2, 3\}, \{2\}, \{2\}, \{2\}],\\
 [\{1, 3\}, \{2, 3\},
\{2, 3\}, \{1\}, \{2\}, \{2\}], \\
[\{2, 3\}, \{2, 3\}, \{2, 3\}, \{1, 2\}, \{1, 2\}, \{1, 2\}],\\
[\{1, 2, 3\}, \{1, 2, 3\}, \{1, 2, 3\}, \{\}, \{\}, \{\}],\\
 [\{1, 2, 3\}, \{1, 2, 3\}, \{1, 2, 3\}, \{1\}, \{2\}, \{2\}].\\
\end{flalign*}

Observe that  the element $[\{1, 3\}, \{2, 3\},
\{2, 3\},\{1\}, \{2\}, \{2\}] $ gives rise to $9$ $Q$-intersecting sets by permuting separately  the elements $[ \{1\}, \{2\}, \{2\}]$ and $[\{1, 3\}, \{2, 3\},
\{2, 3\}].$
So we indeed obtain $38$ elements given by $3\times 8+9+2+3.$
Again,  it is possible to  see directly that we can construct ``by hand" a subrepresentation of a generic $r\in \CH_Q(\CE)$ with the Schubert positions as required.

Let us give an example. Consider $\CK=  [\{2\}, \{3\}, \{3\}, \{2\}, \{2\}, \{2\}]$ and let us see that for generic $r$,
there is a unique subrepresentation $\CS$ of $r$ in this position. Indeed we are forced to take $S_3=r_{4\to 3}(\C^2)\cap r_{1\to 3} (\C^2)$ a subspace of dimension $1$, and we follow $S_3$ by the maps $r_a$ or their inverse.

It is also easy to check the inductive criterium.

\bigskip

{\bf Example {$Q_{3,3}$}} with dimension vector $[3,3,3,3,3,3]$

 The list of elements in $\Int_0(\CA,3,3)$   contains $103$ elements.
 As it is invariant under   permutations in $\Sigma_3\times \Sigma_3$,
   we list only representatives modulo this group, there are $23$ representatives.

\begin{flalign*}
[\{ 1 \}, \{ 3 \}, \{3\}, \{3\},\{3\},\{3\} ],\\
 [\{2\}, \{2\}, \{3\}, \{3\}, \{3\}, \{3\}], \nonumber \\
 [\{ 3 \}, \{ 3 \}, \{3\}, \{1\},\{3\},\{3\} ], \\
 [\{3\}, \{3\}, \{3\}, \{2\}, \{2\}, \{3\}], \nonumber \\
 [\{2, 3\},\{2, 3\}, \{2, 3\}, \{1, 2\}, \{2, 3\}, \{2, 3\}], \\
  [\{2, 3\}, \{2, 3\}, \{2, 3\}, \{1, 3\}, \{1, 3\}, \{2, 3\}],\\
 [\{1, 2\},\{2, 3\}, \{2, 3\}, \{2, 3\}, \{2, 3\}, \{2, 3\}], \\
  [\{1, 3\}, \{1, 3\}, \{2, 3\}, \{2, 3\}, \{2, 3\}, \{2, 3\}],\\
  [ \{1, 2, 3 \}, \{1, 2, 3\}, \{1, 2, 3\}, \{1, 3\}, \{1, 3\}, \{2, 3\} ],  \nonumber \\
 [ \{1\}, \{ 3 \},\{ 3\}, \{\},\{\},\{\}],  \\
  [\{2\}, \{2\}, \{ 3\}, \{\}, \{\}, \{\}], \\
  [\{2\}, \{3\}, \{3\}, \{2\}, \{3\}, \{3\}],\\
   [\{1, 2\}, \{2, 3\},\{2, 3\}, \{\}, \{\},\{\}], \\
[\{1, 2\}, \{2, 3\}, \{2, 3\}, \{2\}, \{3\}, \{3\}], \\
 [\{1, 3\}, \{1, 3\}, \{2, 3\}, \{\}, \{\}, \{\}],\\
 [\{1, 3\}, \{1, 3\}, \{2, 3\}, \{2\}, \{3\},\{3\}], \\
   [\{1, 3\}, \{2, 3\},\{2, 3\}, \{1\}, \{3\}, \{3\}], \\
 [\{1, 3\}, \{2, 3\}, \{2, 3\}, \{2\}, \{2\}, \{3\}],\\
  [\{1, 3\}, \{2, 3\}, \{2, 3\}, \{1, 3\},\{2, 3\},\{2, 3\}], \\
  [\{1, 2, 3\}, \{1, 2, 3\}, \{1, 2, 3\}, \{\}, \{\}, \{\}], \\
   [\{1, 2, 3\}, \{1, 2, 3 \}, \{1, 2, 3 \}, \{1\}, \{3\}, \{3\}], \\
  [\{1, 2, 3\}, \{1, 2, 3 \}, \{1, 2, 3 \}, \{2\}, \{2\}, \{3\}],\\
   [\{1, 2, 3\}, \{1, 2, 3 \}, \{1, 2, 3 \}, \{1,2\}, \{2, 3\}, \{2, 3\}].\\
\end{flalign*}

So indeed we have $103$ elements obtained as $103=16\times 3+9 \times 6+1.$
As in the previous example   it is possible to  see directly that we can construct ``by hand" a subrepresentation of a generic $r\in \CH_Q(\CE)$ with the Schubert positions as required.

Let us give an example. Consider $\CK=    [\{1, 3\}, \{2, 3\},\{2, 3\}, \{1, 3\}, \{2, 3\}, \{2, 3\}]$.
 The set $\CK$ indexes ${\bm \Omega}=[\Omega_1,\ldots ,  \Omega_6]$ with
$\Omega_1=\Omega_4=\{ X \subset \Gr(2, 3 ), \C e_1 \subset X\}$ and $ \Omega_2= \Omega_3= \Omega_5= \Omega_6= \Gr(2,3).$
Indeed $\Omega_1$ is the $B$ orbit of $\C e_1\oplus \C e_3$ and $e_1$ is invariant by $B$, etc...
Let us see that for generic $r$,
there is a unique subrepresentation $\CS=[S_1,S_2,S_3,S_4,S_5,S_6]$ of $r$ with $S_i\in \Omega_i$.
Indeed, we are forced to take $S_3=r_{1\to 3}(e_1)\cup r_{4\to 3} (e_1)$ a subspace of dimension $2$, ($r$ is generic)
 and we follow $S_3$ by the maps $r_a$ or their inverse as before.
%
%
%
%
%
%
%
%
%



\begin{bibdiv}
\begin{biblist}



\bib{quivershort}{article}{
  author={Baldoni, Velleda},
  author={Vergne, Mich\`{e}le},
  author={Walter, Michael},
  title={Horn inequalities and quivers},
  year={2018},
volume={arXiv:1804.00431},
}



\bib{bvw}{article}{
author={Baldoni, Velleda},
author={Vergne, Mich\`ele},
author={Walter, Michael},
title={Horn conditions for quiver subrepresentations and the moment map},
year={2019},
volume={arXiv:1901.07194},
note={To appear in Pure Appl. Math. Q.}
}

\bib{MR2177198}{article}{
  author={Belkale, Prakash},
  title={Geometric proofs of Horn and saturation conjectures},
  journal={J. Algebraic Geom.},
  volume={15},
  date={2006},
  pages={133--173},
}
%





\bib{BV82}{article}{
   author={Berline, Nicole},
   author={Vergne, Mich\`ele},
   title={Fourier transforms of orbits of the coadjoint representation},
   conference={
      title={Representation theory of reductive groups},
      address={Park City, Utah},
      date={1982},
   },
   book={
      series={Progr. Math.},
      volume={40},
      publisher={Birkh\"{a}user Boston, Boston, MA},
   },
   date={1983},
   pages={53--67},
}
\bib{BVW}{article}{
  author={Berline, Nicole},
  author={Vergne, Mich\`{e}le},
  author={Walter, Michael},
  title={The Horn inequalities from a geometric point of view},
  journal={Enseign. Math.},
  volume={63},
  pages={403--470},
  year={2017},
}		
\bib{Deltour}{article}{
   author={Deltour, Guillaume},
   title={Kirwan polyhedron of holomorphic coadjoint orbits},
   journal={Transform. Groups},
   volume={17},
   date={2012},
   number={2},
   pages={351--392},
}



\bib{MR1758750}{article}{
  author={Derksen, Harm},
  author={Weyman, Jerzy},
  title={Semi-invariants of quivers and saturation for Littlewood-Richardson coefficients},
  journal={J. Amer. Math. Soc.},
  volume={13},
  date={2000},
  pages={467--479},
}
\bib{MR3646030}{article}{
   author={Duflo, Michel},
   author={Galina, Esther},
   author={Vargas, Jorge A.},
   title={Square integrable representations of reductive Lie groups with
   admissible restriction to ${\rm SL}_2(\R)$},
   journal={J. Lie Theory},
   volume={27},
   date={2017},
   number={4},
}

\bib{MR1095340}{article}{
   author={Duflo, Michel},
   author={Vergne, Mich\`ele},
   title={Orbites coadjointes et cohomologie \'{e}quivariante},
   language={French},
   conference={
      title={The orbit method in representation theory},
      address={Copenhagen},
      date={1988},
   },
   book={
      series={Progr. Math.},
      volume={82},
      publisher={Birkh\"{a}user Boston, Boston, MA},
   },
   date={1990},
   pages={11--60},
}


\bib{GS1982convex}{article}{
  author={Guillemin, Victor},
  author={Sternberg, Shlomo},
  title={Convexity properties of the moment mapping},
  journal={Invent. math.},
  volume={67},
  date={1982},
  pages={419--513},
}

\bib{GS1982qr}{article}{
  author={Guillemin, Victor},
  author={Sternberg, Shlomo},
  title={Geometric quantization and multiplicities of group representations},
  journal={Invent. math.},
  volume={67},
  date={1982},
  pages={515--538},
}	

\bib{MR1337191}{article}{
   author={Hilgert, Joachim},
   author={Neeb, Karl-Hermann},
   author={\O rsted, Bent},
   title={The geometry of nilpotent coadjoint orbits of convex type in
   Hermitian Lie algebras},
   journal={J. Lie Theory},
   volume={4},
   date={1994},
   number={2},
   pages={185--235},
   issn={0949-5932},
   review={\MR{1337191}},
}


\bib{MR0140521}{article}{
  author={Horn, Alfred},
  title={Eigenvalues of sums of Hermitian matrices},
  journal={Pacific J. Math.},
  volume={12},
  date={1962},
  pages={225--241},
}



\bib{JakobsenVergne}{article}{
   author={Jakobsen, Hans Plesner},
   author={Vergne, Mich\`ele},
   title={Restrictions and expansions of holomorphic representations},
   journal={J. Functional Analysis},
   volume={34},
   date={1979},
   number={1},
   pages={29--53},
}


\bib{MR1671451}{article}{
  author={Knutson, Allen},
  author={Tao, Terence},
  title={The honeycomb model of $\GL_n(\C)$ tensor products. I. Proof of the saturation conjecture},
  journal={J. Amer. Math. Soc.},
  volume={12},
  date={1999},
  pages={1055--1090},
}

\bib{NessMumford84}{article}{
  author={Ness, Linda},
  title={A stratification of the null cone via the moment map},
  note={With an appendix by David Mumford},
  journal={Amer. J. Math.},
  volume={106},
  date={1984},
  pages={1281--1329},
}

\bib{MR639200}{article}{
   author={Olshanski, Grigorii Iosifovich},
   title={Invariant cones in Lie algebras, Lie semigroups and the
   holomorphic discrete series},
   language={Russian},
   journal={Funktsional. Anal. i Prilozhen.},
   volume={15},
   date={1981},
   number={4},
}


\bib{MR3349304}{article}{
   author={Paradan, Paul-Emile},
   title={Quantization commutes with reduction in the non-compact setting:
   the case of holomorphic discrete series},
   journal={J. Eur. Math. Soc. (JEMS)},
   volume={17},
   date={2015},
   number={4},
   pages={955--990},
}

\bib{Hornpq}{article}{
author={Paradan, Paul-Emile},
title={Horn(p,q)},
year={2020},
volume={arXiv:2006.08989},
}


\bib{Repka}{article}{
   author={Repka, Joe},
   title={Tensor products of holomorphic discrete series representations},
   journal={Canadian J. Math.},
   volume={31},
   date={1979},
   number={4},
   pages={836--844},
}





\bib{ressayre2010geometric}{article}{
  title={Geometric invariant theory and the generalized eigenvalue problem},
  author={Ressayre, Nicolas},
  journal={Invent. Math.},
  volume={180},
  pages={389--441},
  year={2010},
}
%
%



\bib{ressayrepc}{unpublished}{
  title={private communication},
  author={Ressayre, Nicolas},
  year={2018}}
%
\bib{MR1162487}{article}{
  author={Schofield, Aidan},
  title={General representations of quivers},
  journal={Proc. London Math. Soc. (3)},
  volume={65},
  date={1992},
  pages={46--64},
}




	
%




\bib{VW}{article}{
  author={Vergne, Mich\`ele},
  author={Walter, Michael},
  title={Inequalities for Moment Cones of Finite-Dimensional Representations},
  journal={J. Symplectic Geom.},
  volume={15},
  date={2017},
  pages={1209--1250},
}

\bib{MR565090}{article}{
   author={Vinberg, Ernest Borisovitch},
   title={Invariant convex cones and orderings in Lie groups},
   language={Russian},
   journal={Funktsional. Anal. i Prilozhen.},
   volume={14},
   date={1980},
   number={1},
   pages={1--13, 96},
}









%
%
%




\end{biblist}
\end{bibdiv}
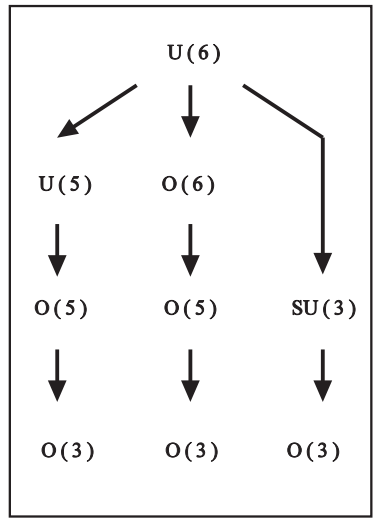
\end{document}